\documentclass[a4paper,12pt]{article}
\usepackage{amssymb}
\usepackage{amsfonts}
\usepackage{amsmath}
\usepackage{amsthm}

\def\CMPsh{Commun.~Math.~Phys.}

\def\es{\emptyset}
\def\eps{\varepsilon}

\def\R{{\mathbb R}}

\def\Z{{\mathbb Z}}

\def\cA{{\mathcal{A}}}
\def\cC{{\cal C}}

\def\cF{{\cal F}}
\def\cG{{\cal G}}
\def\cN{{\cal N}}
\def\cP{{\cal P}}
\def\cR{{\cal R}}
\def\cS{{\cal S}}
\def\cT{{\cal T}}
\def\cX{{\cal{X}}}

\def\dist{\mathrm{dist}}
\def\length{\mathrm{length}}
\def\diam{\mathrm{diam}}
\def\deg{\mathrm{deg}}
\def\AUT{\mathrm{AUT}}

\def\Om{\Omega}

\def\bt{\mathbf{t}}

\def\LE{\mathrm{LE}}
\def\P{\mathbf{P}}

\def\Unif{\mathrm{Unif}}
\def\la{\lambda}

\newcommand{\eqn}[2]{\begin{equation}\label{#1}#2\end{equation}}
\newcommand{\eqnst}[1]{\begin{equation*}#1\end{equation*}}
\newcommand{\eqnspl}[2]{\begin{equation}\begin{split}\label{#1}%
    #2\end{split}\end{equation}}
\newcommand{\eqnsplst}[1]{\begin{equation*}\begin{split}%
    #1\end{split}\end{equation*}}

\theoremstyle{plain}
\newtheorem{theorem}{Theorem}
\newtheorem{lemma}{Lemma}
\newtheorem{proposition}{Proposition}
\newtheorem{corollary}{Corollary}

\newtheorem{question}{Open question}

\theoremstyle{definition}
\newtheorem{definition}{Definition}
\newtheorem{assumption}{Assumption}

\theoremstyle{remark}
\newtheorem{remark}{Remark}

\begin{document}

\title{Abelian sandpiles: an overview and 
results on certain transitive graphs}
\author{Antal A.~J\'arai}

\maketitle

\begin{abstract}
We review the Majumdar-Dhar bijection between recurrent
states of the Abelian sandpile model and spanning trees.
We generalize earlier results of Athreya and J\'arai 
on the infinite volume limit of the stationary distribution
of the sandpile model on $\Z^d$, $d \ge 2$, to a large 
class of graphs. This includes: (i) graphs on which the
wired spanning forest is connected and has one end;
(ii) transitive graphs with volume growth at least 
$c n^5$ on which all bounded harmonic functions are
constant.
We also extend a result of Maes, Redig and Saada on 
the stationary distribution of sandpiles on infinite
regular trees, to arbitrary exhaustions.
\end{abstract}

\section{Introduction}

This paper is based on a talk given at an IRS meeting
in Paris\footnote{Inhomogeneous Random Systems, 
\emph{Stochastic Geometry and Statistical Mechanics}, 
Institut Henri Poincar\'e, Paris, 27 January 2010.}, 
and contains most of the results discussed in the talk, 
with proofs. We give an overview of the Abelian sandpile 
model, with particular emphasis on the Majumdar-Dhar bijection.
Then we discuss recent results on the infinite volume limit of
the model on certain transitive graphs.

The Abelian sandpile model and close variants were
discovered independently in various contexts. Our focus 
here will be the context of probability models 
on graphs; see the references in \cite{HLMPPW} 
for surprising connections with other fields of mathematics.
``Sandpile'' models were introduced by 
Bak, Tang and Wiesenfeld \cite{BTW87} as simple toy
examples, in an attempt to explain the physical mechanisms
underlying the widespread occurrence of power-law 
distributions and fractals in nature.
They introduced the idea of self-organized criticality (SOC)
as a possible mechanism, and studied ``sandpile'' models
numerically to support their claims.
The importance of the model as a theoretical tool to
study SOC was recognized by Dhar \cite{Dhar90}, who 
generalized it and discovered some of its fundamental 
properties, including the Abelian property. Dhar coined the 
name ``Abelian sandpile''. The definition of the model 
is recalled in Section \ref{sec:model}, and the key
results needed are summarized in Section \ref{sec:MDbij}.
For further background on the mathematical results, 
see the survey by Redig \cite{Redig}. See also the paper
by Holroyd, Levine, M\'esz\'aros, Peres, Propp and Wilson
\cite{HLMPPW} for a rigorous and self-contained introduction 
as well as an account of the connection of sandpiles with the 
rotor-router model. 
The paper \cite{HLMPPW} also contains extensions to 
directed graphs of some of the results discussed in 
Sections \ref{sec:model} and \ref{sec:MDbij}.
See the survey by Dhar \cite{Dhar06} for the 
theoretical physics context.

Our main focus will be the following type of question.
Let $G = (V, E)$ be an infinite, locally finite
graph, for example $\Z^d$, or the Cayley graph of a finitely
generated discrete group. Let $V_1 \subset V_2 \subset \dots \subset V$
be a sequence of finite subsets such that 
$\cup_{n=1}^\infty V_n = V$. Do the Abelian 
sandpile models on the $V_n$'s converge to a limiting model
on $V$? 

The above question was first addressed in the case of
$\Z$ by Maes, Redig, Saada and Van Moffaert \cite{MRSvM}. 
Here the limiting model has a 
trivial stationary distribution, nevertheless the question
of convergence to this distribution is non-trivial \cite{MRSvM}.

Maes, Redig and Saada \cite{MRS02} considered sandpiles on
infinite regular trees.
A stationary Markov process was constructed, obtained as 
the limit of sandpile Markov chains on finite subgraphs.
For the most part, the construction given in \cite{MRS02}
is very general, and applies to a general infinite 
graph. There were two key steps, however, that were 
specific to the tree: (i) to show that the stationary 
measures of sandpiles on a suitable sequence of finite 
subgraphs converge weakly to a unique (automorphism 
invariant) limit; (ii) to show that avalanches are 
almost surely finite in the limit. These steps were 
carried out making use of results of Dhar and 
Majumdar \cite{DM90}.

Maes, Redig and Saada \cite{MRS04} studied a 
so-called \emph{dissipative} 
version of the $\Z^d$ model, where particles are removed
on each toppling (not only at the boundary). The presence
of dissipation introduces fast decay of correlations.
Making use of this, the steps (i)--(ii) above could be 
carried out, and the infinite volume process was constructed. 
A nice feature of the limiting process is that it is shown 
to live on a compact Abelian group, extending the finite 
volume formalism.

For the usual (non-dissipative) model on $\Z^d$, the step 
(i) above for $d \ge 2$ was solved 
by Athreya and J\'arai \cite{AJ04}, and step (ii) 
for $d \ge 3$ was solved by J\'arai and Redig \cite{JR08}. 
The main new ingredient in these papers was to exploit a 
result of Majumdar and Dhar \cite{MD92} that gives a 
bijection between the recurrent states of the sandpile model
and wired spanning trees of the underlying graph.
This made it possible to use techniques from the theory 
of uniform spanning trees, in particular Pemantle's 
theorem \cite{Pem} on the existence of the wired uniform
spanning forest. There is a difference between the cases
$2 \le d \le 4$ and $d \ge 5$, that are a reflection 
of Pemantle's result that in the former case the 
spanning forest is a.s.~connected, while in the latter 
case it is not. Another essential ingredient is that 
the each tree in the spanning forest has one end. This was
first proved by Pemantle \cite{Pem} and 
Benjamini, Lyons, Peres and Schramm
\cite{BLPS01}. (In what follows, we will abreviate the
latter authors to BLPS.)

The surveys \cite{MRS05} and \cite{Jarai05} discuss the
above developments. 

More recently, a continuous height dissipative model was 
studied on $\Z^d$ by J\'arai, Redig and Saada \cite{JRS10}. 
This extends the discrete dissipative model considered 
in \cite{MRS04} by allowing the amount dissipated per toppling 
to be any non-negative real value, rather than an
integer. This has the advantage that the limit of
zero dissipation can be formulated precisely. In this limit, 
the discrete non-dissipative model is recovered. This work is 
also based on an adaptation of the Majumdar-Dhar 
bijection.

In the present paper, we extend some of the $\Z^d$ results 
to more general graphs. Part of our motivation is the 
well-known fact that Pemantle's argument \cite{Pem}  
(made explicit by H\"aggstr\"om \cite{Hagg}) shows 
that the wired spanning forest measure exists on any 
infinite locally finite graph, as a limit from uniform 
spanning trees on finite graphs (see Theorem \ref{thm:Pem} 
below). Also, the alternative distinguishing a connected 
spanning forest from the disconnected case can be vastly 
generalized \cite[Theorems 9.2, 9.4]{BLPS01}. 
At first sight, the Majumdar-Dhar bijection may suggest 
that a general convergence statement should also exist 
for sandpiles and could be derived from the bijection. 
However, the situation is more subtle. 
Although the wired spanning forest measure is always
unique, the limits of sandpile measures may be non-unique.
Results of J\'arai and Lyons \cite{JL07} 
(see Theorem \ref{thm:JL} below) show that this is the
case for a class of graphs with two ends, on which the
wired spanning forest has two ends a.s.
After making appropriate assumptions to exclude the above
phenomenon, a general convergence statement can be proved 
for certain ``low-dimensional'' graphs. One can follow essentially 
the same argument as the one made in \cite{AJ04} for 
$\Z^d$, $2 \le d \le 4$. Nevertheless, we decided to 
include a new proof in the present paper, 
that follows a somewhat different
route. The argument we present is based on coupling, 
and hence gives more than weak convergence; 
see Theorem \ref{thm:AJlow} below. 

Considerably more
work is needed to extend the line of argument made in 
\cite{AJ04} for $\Z^d$, $d \ge 5$, to a more general
class of transitive graphs. Here we need to make more 
restrictive assumptions on the graph to make the proof work, 
see Assumption \ref{ass:Gcond} and Theorem \ref{thm:trans}. 
Nevertheless, parts of our argument for this case are still
quite general, and are potentially useful beyond the 
validity of Assumption \ref{ass:Gcond}; 
see Lemmas \ref{lem:cycle} and 
\ref{lem:goodblocks} and Proposition \ref{prop:fluct}.
As discussed in Section \ref{sec:limitsII}, results of 
BLPS \cite{BLPS01} and Lyons, Morris and Schramm \cite{LMS08}
imply that there are many graphs on which our assumptions 
are satisfied.

Our results make it possible to apply the general 
machinery developed by Maes, Redig and Saada \cite{MRS02} 
to a large class of graphs. This is outlined in
Section \ref{sec:conclusion}.

The bijection is also useful on infinite regular trees.
We discuss some interesting symmetry properties of the 
bijection on trees, and in Theorem \ref{thm:tree} we use
them to extend the convergence result of \cite{MRS02} to a 
general exhaustion. 

The outline of the paper is as follows.
In Section \ref{sec:model} we review the definition
and basic properties of the Abelian sandpile on
a finite undirected multigraph. Section \ref{sec:MDbij} 
is devoted to a discussion of the Majumdar-Dhar bijection 
that establishes a one-to-one mapping between recurrent
sandpile configurations and spanning trees.
In Section \ref{sec:WSF} we recall the wired spanning 
forest measure and Pemantle's alternative distinguishing
the cases $2 \le d \le 4$ and $d \ge 5$ for $\Z^d$. In 
Section \ref{sec:Wilson} we recall Wilson's method 
and its extensions to infinite graphs. 
In Section \ref{sec:limitsI}, we state and prove the
general convergence theorem for ``low-dimensional'' 
graphs. In Section \ref{sec:limitsII} we state and prove 
a convergence theorem for certain ``high-dimensional''
transitive graphs. In Section \ref{sec:limitsIII}, 
we discuss the results for regular trees. 
Finally, in Section \ref{sec:conclusion} we 
make the connection with the results of Maes, Redig
and Saada \cite{MRS02}.

\medbreak

{\bf Acknowledgments.} I thank Fran\c{c}ois Dunlop, 
Thierry Gobron and Ellen Saada for giving me the
opportunity to present this talk. I am grateful for a 
question of David B.~Wilson that prompted me to consider 
the problems in Section \ref{sec:limitsIII}.

\section{The Abelian sandpile model}
\label{sec:model}

We first define the model on a finite, connected multigraph 
$G = (V^+,E)$ that has a distinguished vertex $s$,
called the \emph{sink}. We write $V = V^+ \setminus \{ s \}$.
We allow $G$ to have loop-edges, as this has no 
major consequence.
We write $a_{xy} = a_{yx}$ for the number of edges 
between $x$ and $y$ in $G$, where $x, y \in V^+$.
Sometimes we will consider the directed graph 
$\vec{G} = (V^+, \vec{E})$ that is obtained from 
$G$ by replacing each edge by two directed edges,
one in each direction. A directed edge $e \in \vec{E}$, 
points from the \emph{tail of $e$}, denoted $e_-$, 
to the \emph{head of $e$}, denoted $e_+$.
When there is no ambiguity, we also write 
$e = [e_-,e_+]$, to specify an oriented edge by its 
tail and head. 
For a set of vertices $A \subset V^+$, we denote 
by $G \setminus A$ the graph obtained by removing
all vertices in $A$ from $G$, as well as all
edges incident with vertices in $A$.

We define the set of \emph{stable configurations} of
particles:
\eqnst
{ \Omega_G 
  := \prod_{x \in V} \{ 0, \dots, \deg_G(x)-1 \}, }
and the set of all particle configurations:
\eqnst
{ \cX_G 
  := \prod_{x \in V} \{ 0, 1, \dots \}. }
The dynamics of the model is defined in terms of
the \emph{toppling matrix}, that is the graph Laplacian:
\eqnst
{ \Delta_{xy}
  = (\Delta_G)_{xy}
  = \begin{cases}
    \deg_G(x) - a_{xx} & \text{if $x=y$;}\\
    -a_{xy}            & \text{if $x \not= y$.}\\
    0  & \text{otherwise.}
    \end{cases} \qquad x, y \in V. }
We define the basic operation of \emph{toppling}.
If $\eta \in \cX_G$ and $\eta_x \ge \deg_G(x)$, then
$x$ is allowed to \emph{topple}, which means that it
sends one particle along each edge incident with $x$.
This can be written:
\eqnst
{ \eta_y \longrightarrow \eta_y - \Delta_{xy}, \quad 
  y \in V, }
that is, the row of $\Delta_G$ corresponding to $x$ 
is subtracted from the configuration $\eta$.
Note that if $x$ was allowed to topple, 
the new configuration is also in $\cX_G$.
When $x$ is a neighbour of $s$, $a_{xs}$ particles 
are lost as the result of toppling, 
otherwise the number of particles is conserved.

We define the \emph{stabilization map}
$\cS : \cX_G \to \Omega_G$, by applying topplings
as long as possible. It can be shown that
any configuration stabilizes in finitely many 
steps, and the resulting stable configuration 
is independent of the sequence of topplings used.
This is summarized in the following lemma.

\begin{lemma}[{Dhar \cite{Dhar90}; see 
also \cite[Lemma 2.2, Lemma 2.4]{HLMPPW}}]
$\cS$ is well-defined.
\end{lemma}

We define the \emph{addition operators}
$a_x : \Omega_G \to \Omega_G$ by
$a_x \eta := \cS ( \eta + \delta_{x,\cdot} )$, $x \in V$, where
$\delta_{x,y} = 1$ if $y = x$ and $=0$ when $y \not= x$.
The addition operators satisfy the \emph{Abelian property}:
$a_x a_y = a_y a_x$, $x,y \in V$ \cite{Dhar90}; see also
\cite[Lemma 2.5]{HLMPPW}. 
Let $\{ p(x) \}_{x \in V}$ be a distribution on $V$
satisfying $p(x) > 0$, $x \in V$. 
We define a Markov chain with statespace $\Om_G$, 
where a single step
consists of picking a vertex $x \in V$ at random, 
according to the distribution $p$, 
and applying $a_x$ to the configuration.
The set of \emph{recurrent configurations} $\cR_G$
is the set of recurrent states of the Markov chain.
The \emph{Sandpile Group} of $G$ is defined as
\eqnst
{ K_G 
  := \Z^{V} / \Z^{V} \Delta_G, }
that is, $\Z^{V}$ factored by the integer row span
of $\Delta_G$. 

\begin{theorem}[{Dhar \cite{Dhar90}; 
see also \cite[Corrolary 2.16]{HLMPPW}}]\ \\
(i) The restriction of the map $a_x$ to $\cR_G$ is a 
one-to-one transformation of $\cR_G$ onto itself,
for each $x \in V$. These restricted maps 
generate an Abelian group isomorphic to $K_G$.\\
(ii) $|\cR_G| = |K_G| = \det(\Delta_G)$.\\
(iii) The Markov chain has a unique stationary 
distribution $\nu_G$ and this is the uniform 
distribution on $\cR_G$.
\end{theorem}

By the Matrix-Tree Theorem \cite[Theorem II.12]{Bollobas},
$\det(\Delta_G)$ also equals the number of spanning trees of
$G$. Let us write $\cT_G$ for the set of all spanning trees
of $G$. It is natural to ask for an explicit bijection between
$\cR_G$ and $\cT_G$, and such a bijection is discussed
in Section \ref{sec:MDbij}. See \cite{HLMPPW} for 
a different class of bijections, based on the
rotor-router walk.

\section{The Majumdar-Dhar bijection}
\label{sec:MDbij}

In this section we describe our main tool for
studying infinite volume limits of sandpiles.
Let $G = (V^+,E)$ be a finite, connected multigraph, 
and $s$ the distinguished vertex (the sink). 
Recall that $\nu_G$ is the stationary distribution,
$\cR_G$ is the set of recurrent configurations, and
$\cT_G$ is the set of spanning trees of $G$.
We describe a bijection between $\cR_G$ and $\cT_G$
that was introduced by Majumdar and Dhar \cite{MD92}.

\subsection{Allowed configurations}
\label{ssec:allowed}

For a subset $F \subset V$ and $x \in F$, we write 
$\deg_F(x) = \sum_{y \in F} a_{yx}$, which is the degree
of $x$ in the subgraph induced by $F$. We write $\eta_F$ for
the restriction of the configuration $\eta$ to the subset $F$.
We say that $\eta_F$ is a \emph{forbidden subconfiguration (FSC)} 
if for all $x \in F$, $\eta_x < \deg_F(x)$. 
We say that $\eta \in \Omega_G$ is 
\emph{allowed}, if there is no $F \subset V$, $F \not= \es$,
such that $\eta_F$ is a FSC. Let us write $\cA_G$ for the 
set of allowed configurations. In Section \ref{ssec:burn}
we review Dhar's Buring Algorithm that decides if a 
given configuration is allowed or not.

It was proved in \cite{Dhar90} that $\cR_G \subset \cA_G$. 
It was proved in \cite{MD92}, with the introduction of the 
bijection in Section \ref{ssec:bijection}, 
that $|\cA_G| = |\cT_G|$. Hence it follows that 
$|\cR_G| = \det(\Delta_G) = |\cT_G| = |\cA_G|$, and 
therefore $\cR_G = \cA_G$. See \cite[Lemma 4.2]{HLMPPW}
for a different proof of the latter fact, that is still 
based on the Burning Algorithm, but does not require 
the bijection.

\begin{lemma}[{Dhar \cite{Dhar90}; Majumdar, Dhar \cite{MD92}; 
see also \cite[Lemma 4.2]{HLMPPW}}]
Suppose that $G$ is a connected multigraph with a sink $s$
specified. Then $\cR_G = \cA_G$.
\end{lemma}

\subsection{The Burning Algorithm}
\label{ssec:burn}

The following algorithm, introduced by Dhar \cite{Dhar90}, 
checks if a configuration is allowed.
Let $\eta \in \Omega_G$.
Set $B(0) := \{ s \}$, and $U(0) = V$. 
For $i \ge 1$, we inductively define
\eqnsplst
{ B(i)
  &:= \{ x \in U(i-1) :
     \eta_x \ge \deg_{U(i-1)}(x) \} \\
  U(i)
  &:= U(i-1) \setminus B(i)
  = V \setminus \left( \cup_{j=0}^i B(j) \right). }
We call $B(i)$ the \emph{set of vertices burning 
at time $i$}, and $U(i)$ the \emph{set of 
unburnt vertices at time $i$}.
We say that the algorithm \emph{terminates}, if for some
$i \ge 1$ we have $U(i) = \es$. It is easy to prove by induction
on $i$ that for all $i \ge 1$, no vertex in $B(i)$ 
can be part of any FSC. It also follows from the 
definition of $B(i)$ that if for some $i \ge 1$
we have $B(i) = \es$ and $U(i-1) \not= \es$ 
(i.e.~the algorithm does not terminate), 
then $\eta_{U(i-1)}$ is an FSC. Hence the algorithm 
terminates if and only if $\eta$ is allowed.

This algorithm can be generalized to Eulerian
digraphs with a sink; see \cite[Lemma 4.1]{HLMPPW}
The algorithm does not work on general directed
graphs. An extension to that case, called
the \emph{script algorithm}, was given by
Speer \cite{Speer}.

\subsection{The bijection}
\label{ssec:bijection}

Based on the Burning Algorithm, we now give the 
bijection between $\cA_G$ and $\cT_G$. 
The bijection is not canonical, in the sense that 
some choices can be made how to set it up. Suppose that 
for every $x \in V$, every non-empty
$P \subset \{ e \in \vec{E} : e_- = x \}$
and every finite $K \subset \{0, 1, 2, \dots, \deg_G(x)-1 \}$ 
of the form $K  = \{ j, j + 1, \dots, j + |P|-1 \}$, an arbitrary 
bijection $\alpha_{P,K} : P \to K$ is fixed. Then the
bijection between $\cA_G$ and $\cT_G$ will be uniquely 
defined in terms of the $\alpha_{P,K}$'s.

We define the map $\sigma_G : \cA_G \to \cT_G$. Let $\eta \in \cA_G$,
and consider the sets $\{ B(i) \}_{i \ge 0}$ defined in Section \ref{ssec:burn}.
By the definition of the Burning Algorithm we have $V = \cup_{i \ge 1} B(i)$, 
and this is a disjoint union. We build the tree $t = \sigma_G(\eta)$
by connecting a vertex $x \in B(i)$, $i \ge 1$ to some vertex in $B(i-1)$.
This ensures that there are no loops, and since $V = \cup_{i \ge 1} B(i)$,
$t$ will be a spanning tree of $G$.
Suppose then that $x \in B(i)$ for some $i \ge 1$. 
Let 
\eqnspl{e:data}
{ n_x 
  &:= \sum_{y \in \cup_{j=0}^{i-1} B(j)} a_{yx} \\
  P_x 
  &:= \{ e \in \vec{E} : e_+ \in B(i-1),\, e_- = x \} \\
  K_x
  &= \{ \deg_G(x) - n_x, \dots, \deg_G(x) - n_x + |P_x|-1 \}. }
We claim that $\eta_x \in K_x$. For this, note that 
due to $x \in B(i)$ we have
\eqnst
{ \eta_x 
  \ge \deg_{U(i-1)}(x) 
  = \sum_{y \in U(i-1)} a_{yx}
  = \deg_G(x) - \sum_{y \in \cup_{j = 0}^{i-1} B(j)} a_{yx}
  = \deg_G(x) - n_x. }
On the other hand, we have 
$|P_x| = \sum_{y \in B(i-1)} a_{xy} = \sum_{y \in B(i-1)} a_{yx}$,
and since $x \not\in B(i-1)$, for $i \ge 2$ we have
\eqn{e:inK}
{ \eta_x
  < \deg_{U(i-2)}(x)
  = \deg_G(x) - \sum_{y \in \cup_{j=0}^{i-2} B(j)} a_{yx}
  = \deg_G(x) - n_x + |P_x|. }
When $i = 1$, we have $n_x = |P_x|$, so we still have 
$\eta_x < \deg_G(x) - n_x + |P_x|$.
This shows that indeed $\eta_x \in K_x$. 
It follows that the edge $e_x := \alpha_{P_x,K_x}^{-1}(\eta_x)$
is an edge pointing from $x$ to a vertex in $B(i-1)$.
If we define 
\eqnst
{ \sigma_G(\eta)
  := t 
  := \{ e_x : x \in V \}, }
then $t$ is a spanning tree of $G$ with each edge 
directed towards $s$, or equivalently, disregrading
the orientedness, a spanning tree of $G$.

\begin{lemma}[Majumdar, Dhar \cite{MD92}]
\label{lem:bijection}
The map $\sigma_G : \cA_G \to \cT_G$ is a bijection
between these sets.
\end{lemma}

\begin{proof}
We first show that $\sigma_G$ is injective. Let 
$\eta^1, \eta^2 \in \cA_G$, $\eta^1 \not= \eta^2$, and
let $t^1 := \sigma_G(\eta^1)$, $t^2 := \sigma_G(\eta^2)$.
Let $i \ge 1$ be the smallest index such that either 
$B(i,\eta^1) \not= B(i,\eta^2)$ or there exists 
$x \in B(i,\eta^1) = B(i,\eta^2)$ with $\eta^1_x \not= \eta^2_x$.
If such index did not exist, we would get by induction
on $i$ that $\eta^1 = \eta^2$ on 
$\cup_{i \ge 1} B(i,\eta^1) = \cup_{i \ge 1} B(i,\eta^2) = V$,
a contradiction. By the choice of $i$, we have
\eqn{e:equality}
{ \text{$B(j,\eta^1) = B(j,\eta^2)$ for $1 \le j \le i-1$.} }
If $B(i,\eta^1) \not= B(i,\eta^2)$, then pick a vertex
$x$ in the symmetric difference. Then by the construction 
of $\sigma_G$, in one of $t^1$ and $t^2$ there is an edge
from $x$ to $B(i-1,\eta^1) = B(i-1,\eta^2)$ 
and there is no such edge in the other, so 
$t^1 \not= t^2$. Suppose therefore
that $B(i,\eta^1) = B(i,\eta^2)$, but there exists
$x \in B(i,\eta^1) = B(i,\eta^2)$ such that $\eta^1_x \not= \eta^2_x$.
By the equality \eqref{e:equality}, we have
$n_{x}(\eta^1) = n_{x}(\eta^2)$, 
$P_{x}(\eta^1) = P_{x}(\eta^2)$, and hence also
$K_{x}(\eta^1) = K_{x}(\eta^2)$. 
However, since $\eta^1_x \not= \eta^2_x$ we have
$\alpha_{P_{x},K_{x}}^{-1}(\eta^1_x) \not= 
\alpha_{P_{x},K_{x}}^{-1}(\eta^2_x)$, and
therefore the edge between $x$ and $B(i-1)$ 
is different in $t^1$ and $t^2$.
This completes the proof of injectivity.

We now show that $\sigma_G$ is surjective. In the course
of doing so, we find the inverse map 
$\sigma_G^{-1} =: \phi_G : \cT_G \to \cA_G$.
First we note that for any $\eta \in \cA_G$, 
the sets $B(0), B(1), \dots$ and the data in \eqref{e:data} 
can be easily expressed in terms of $t = \sigma_G(\eta)$ 
as well. Namely, let $d_t(\cdot,\cdot)$ denote graph 
distance in the tree $t$. Then due to the construction 
of $t$, we have
\eqnspl{e:B's}
{ B(0) 
  &= \{ s \}; \\
  B(i) 
  &= \{ x \in V : d_t(s, x) = i \}, 
    \qquad i \ge 1. }
Since this expresses $B(0), B(1), \dots$ in terms 
of $t$, the formulas \eqref{e:data} show that 
$n_{x}$, $P_{x}$ and $K_{x}$ are also 
expressed in terms of $t$.
Also, by the definition of $\sigma_G$, the unique
edge of $t$ in $P_{x}$ is $e_x$, hence
we have $\eta_x = \alpha_{P_{x}, K_{x}} (e_x)$.

The above makes it clear what the inverse $\phi_G = \sigma_G^{-1}$ 
has to be. Suppose that $t \in \cT_G$ is given.
We use \eqref{e:B's} to \emph{define} the $B(i)$'s and 
for $x \in B_i$, $i \ge 1$, we use \eqref{e:data} 
as the definition of $n_{x}$, $P_{x}$ and $K_{x}$.
It is immediate from these definitions that 
$P_x$ is non-empty, and $t$ has a unique edge
in $P_x$. Therefore, for $x \in B(i)$, $i \ge 1$ 
we let $e_x$ be the unique edge of $t$ in $P_{x}$, 
and we set $\eta_x = \alpha_{P_{x}, K_{x}} (e_x)$.
We define $\phi_G(t) := \eta$. It is clear that 
if $\eta \in \cA_G$, then $\sigma_G(\phi_G(t)) = t$.
What is left to show is that we always have 
$\eta \in \cA_G$.

We prove that for every $t \in \cT_G$ we have
$\eta = \phi_G(t) \in \cA_G$, 
by applying the Burning Test to $\eta$.
By definition, $B(0) = \{ s \}$. We also
set $U(0) = V$, and recursively, 
$U(i) := U(i-1) \setminus B(i)$ for $i \ge 1$.
We show by induction on $i$ that at time
$i \ge 0$ precisely $B(i)$ burns.

We know that at time $0$, $B(0)$ and $U(0)$ are 
the set of burning and unburnt sites.
Suppose inductively that $i \ge 1$ and we already 
know that at time $0 \le j \le i-1$ exactly 
$B(j)$ burns, and hence $U(i-1)$ is the set 
of unburnt sites at time $i-1$. We show that 
at time $i$, precisely $B(i)$ burns.

Let $x \in B(i)$. Then due to the 
inductive hypothesis and the definition of
$n_{x}$, we have
\eqn{e:degree}
{ \deg_{U(i-1)}(x)
  = \sum_{y \in U(i-1)} a_{yx}
  = \sum_{y \in V^+ \setminus \cup_{j=0}^{i-1} B(j)} a_{yx}
  = \deg_G(x) - n_x. }
Since $\eta_x \in K_{x}$ (by the definition of 
$\eta = \phi_G(t)$), we have 
$\eta_x \ge \deg_G(x) - n_x$. Hence due to
\eqref{e:degree}, $x$ burns at time $i$. 

Let now $x \in B(j)$ with $j \ge i+1$. Then by the 
induction hypothesis, $B(j-1), B(j), \dots$ are unburnt 
at time $i-1$, and hence 
\eqnst
{ \deg_{U(i-1)}(x) 
  \ge \sum_{y \in \cup_{k \ge j-1} B(k)} a_{yx}
  = \deg_G(x) - \sum_{y \in \cup_{0 \le k \le j-2} B(k)} a_{yx} 
  = \deg_G(x) - n_x + |P_x|. }
Since $\eta_x \in K_{x}$, we have 
$\eta_x < \deg_G(x) - n_x + |P_x|$, and therefore
$x$ does not burn at time $i$. This shows that 
at time $i$ precisely the set $B(i)$ burns, and
completes the induction. Therefore $\eta$ is allowed,
and we have shown that $\sigma_G$ is a bijection 
between $\cA_G$ and $\cT_G$.
\end{proof}

We define the \emph{uniform spanning tree measure}
$\mu_G$ as the probability measure on $\cT_G$ 
that assigns each $t \in \cT_G$ equal weight.
Lemma \ref{lem:bijection} has the following
important corollary.

\begin{corollary}
\label{cor:uniform}
The stationary measure $\nu_G$ of the Abelian sandpile
on $G$ is the image under $\phi_G$ of the uniform
spanning tree measure $\mu_G$.
\end{corollary}

The next lemma summarizes an observation about the 
nature of the inverse map $\phi_G$ that will be important
for infinite volume limits. For $t \in \cT_G$,
write $d_t(\cdot,\cdot)$ for graph distance 
in the tree $t$. For $x \in V$, write 
$\pi_x$ for the unique self-avoiding path 
from $x$ to $s$ in $t$. Let us write $x \sim y$ 
if there exists an edge in $G$ between $x$ and $y$. 
Let
\eqnst
{ \cN_x 
  = \{ y \in V^+ : \text{$y \sim x$ or $y = x$} \}. }
Let $v_x \in V^+$ be the unique vertex such that
$v_x \in \pi_y$ for all $y \in \cN_x$, and
$d_t(s,v_x)$ is maximal. (Informally, this is 
the ``first meeting point'' of the paths 
$\{ \pi_y \}_{y \in \cN_x}$.) Let us write
$\vec{F}_x$ for the following directed subtree of $t$:
\eqnst
{ \vec{F}_x 
  := \left\{ e \in t : 
     e_- \in \cup_{y \in \cN_x} \pi_y,\, 
     d_t(s,e_+) \ge d_t(s,v_x) \right\}. }
Each edge in $\vec{F}_x$ is directed towards $v_x$,
so specifying $\vec{F}_x$ is equivalent to specifying
the undirected, rooted tree $(F_x,v_x)$.
Recall that $e_x$ is the unique edge 
of $t$ satisfying $(e_x)_- = x$ and
$d_t(s,(e_x)_+) = d_t(s,x)-1$.
Write $\eta = \phi_G(t)$.

\begin{lemma}
\label{lem:depend}\ \\
(i) The value of $\eta_x$ only depends on $t$ through the 
rooted subtree $(F_x,v_x)$. \\
(ii) The value of $\eta_x$ only depends on $t$ through 
the differences $\{ d_t(s,x) - d_t(s,y) : y \sim x \}$ and
the edge $e_x$.\\
(iii) In fact, $\eta_x$ only depends on 
the cardinality of the set $\{ y \sim x : d_t(s,x) - d_t(s,y) \ge 1 \}$,
the set $\{ y \sim x : d_t(s,x) - d_t(s,y) = 1 \}$ and 
the edge $e_x$.
\end{lemma}

\begin{proof}
(i) By the definition of $\phi_G$, $\eta_x$ only 
depends on $n_x$, $P_x$ (which determine $K_x$) and
$e_x$. Due to the characterization of the $B(i)$'s in terms
of graph distance \eqref{e:B's} we have
\eqnsplst
{ n_x
  &= \sum_{\substack{y : y \sim x \\ d_t(s,y) < d_t(s,x)}} a_{yx}
  = \sum_{\substack{y : y \sim x \\ d_t(v_x,y) < d_t(v_x,x)}} a_{yx}
  = \sum_{\substack{y : y \sim x \\ d_{F_x}(v_x,y) < d_{F_x}(v_x,x)}} a_{yx}, }
and the last expression only depends on $(F_x,v_x)$.
Similarly, 
\eqnsplst
{ P_x
  &= \{ e \in \vec{E} : 
     e_- = x,\, d_t(s,e_+) = d_t(s,x) - 1 \} \\
  &= \{ e \in \vec{E} : 
     e_- = x,\, d_{F_x}(v_x,e_+) = d_{F_x}(v_x,x) - 1 \}, }
and the last expression only depends on $(F_x,v_x)$.
Finally, since $e_x$ is the unique edge of $t$ 
incident with $x$ that is directed away from $x$,
we have
\eqnst
{ \text{$e_x$ is the unique edge $e$ in $\vec{F}_x$ such that 
    $e_- = x$ and $d_{F_x}(v_x,e_+) = d_{F_x}(v_x,x) - 1$}. }
(ii) This is similar to part (i). We have
\eqnsplst
{ n_x
  &= \sum_{\substack{y : y \sim x \\ d_t(s,x) - d_t(s,y) > 0}} a_{yx} \\
  P_x
  &= \{ e \in \vec{E} : 
     e_- = x,\, d_t(s,x) - d_t(s,e_+) = 1 \}. }
This proves the claim. (iii) also follows from the above
expressions.
\end{proof}

\section{The Wired Spanning Forest}
\label{sec:WSF}

Let now $G = (V,E)$ be an infinite locally finite
graph. For simplicity, from now on we restrict our attention 
to simple graphs (no multiple edges or loops), but note 
that it is possible to extend all our results 
in Sections \ref{sec:limitsI} and \ref{sec:limitsII} 
to multigraphs, with essentially the same arguments.

An \emph{exhaustion} of $V$ is a sequence
$V_1 \subset V_2 \subset \dots \subset V$ 
such that $\cup_{n=1}^\infty V_n = V$.
Let $G_n = (V^+_n,E_n)$ denote the graph obtained
from $G$ by identifying all vertices in 
$V \setminus V_n$ to a single vertex $s$,
so that $V^+_n = V_n \cup \{ s \}$, and
removing loops at $s$. Sometimes $G_n$ is called
the \emph{wired graph} associated to $V_n$, 
where ``wired'' refers to the fact that 
all connections outside $V_n$ have been forced
to occur. Recall that $\mu_{G_n}$
is the uniform probability measure on the
set of spanning trees $\cT_{G_n}$.
We will write $\Rightarrow$ to denote
weak convergence of measures.

The usefulness for infinite volume limits 
of the bijection in Section \ref{ssec:bijection}
lies in the well-known theorem stated below.
This theorem is implicit in the work of 
Pemantle \cite{Pem}, and was made explicit 
by H\"aggstr\"om \cite{Hagg}, in the case 
of $\Z^d$. The $\Z^d$ proof immediately applies
in the generality stated. 

\begin{theorem}[{Pemantle \cite{Pem}; see also \cite{Hagg}}]
\label{thm:Pem}
Let $G = (V,E)$ be an infinite locally finite graph. 
There exists a measure $\mu$ on $\{ 0, 1 \}^E$ 
such that $\mu_{G_n} \Rightarrow \mu$ 
independently of the exhaustion.
The measure $\mu$ concentrates on spanning forests 
of $G$ all of whose components are infinite.
\end{theorem}

The measure $\mu$ is also called the 
\emph{Wired Spanning Forest (WSF) measure}. 
Theorem \ref{thm:Pem} naturally leads to the following question.

\begin{question}
\label{open:general limit}
Assume the same conditions as in Theorem \ref{thm:Pem}.
Under what extra conditions does $\nu_{G_n}$ have a unique 
weak limit $\nu$ on the space 
$\prod_{x \in V} \{ 0, \dots, \deg_G(x) - 1 \}$,
independently of the exhaustion?
\end{question}

It is not possible to deduce a general convergence statement
only from Theorem \ref{thm:Pem}. On certain graphs with 
two ends the limit is not unique; see Theorem \ref{thm:JL} in 
Section \ref{sec:limitsI}. However, as we will see in 
Section \ref{sec:limitsI}, there is a general convergence
theorem on certain ``low-dimensional'' graphs. We will need 
to consider the number of components of the WSF, and 
the \emph{ends} of the components. We say that an infinite tree 
\emph{has one end}, if any two infinite self-avoiding paths
in the tree have infinitely many vertices in common.
In the theorem below, statement (i) and the first part of
statement (ii) are due to Pemantle \cite{Pem}. 
The statement on one end in part (ii) was first
proved by BLPS \cite{BLPS01} and in much greater generality. 
Lyons, Morris and Schramm \cite{LMS08} gave a simpler 
and even more general proof with quantitative estimates.

\begin{theorem}[{Pemantle \cite{Pem}; BLPS \cite{BLPS01}}]
\label{thm:Pemalt}
Let $G$ be the $\Z^d$ lattice.\\
(i) Suppose $2 \le d \le 4$. The Wired Spanning Forest is 
$\mu$-a.s.~connected, and has one end.\\
(ii) Suppose $d \ge 5$. The Wired Spanning Forest 
$\mu$-a.s.~consists of infinitely many trees, 
and each tree has one end. 
\end{theorem}

\section{Wilson's method}
\label{sec:Wilson}

In this section we recall some facts about Wilson's method, 
that is an indispensable tool in studying uniform spanning 
trees. 

Let $\pi = [\pi_0, \pi_1, \dots, \pi_M]$ be a finite path 
in some graph. The \emph{loop-erasure} of $\pi$
is defined by chronologically removing loops from the path 
as they are created. That is, we set 
$\sigma = \LE(\pi) := [\sigma_0, \dots, \sigma_K]$,
where we inductively define
\eqnsplst
{ s_0 
  &:= 0 \\
  \sigma(0) 
  &:= \pi(0) \\
  s_j 
  &:= \max \{ k \ge s_{j-1} 
      :  \pi(k) = \sigma(j-1) \}, \quad j \ge 1, \\
  \sigma(j) 
  &:= \pi(s_j+1), \quad j \ge 1. }
Note that loop-erasure also makes sense for an infinite
path that visits any vertex only finitely often.

Suppose now that $G = (V,E)$ is a finite graph,
and $w : E \to (0,\infty)$ is a function.
We call $w(e)$ the \emph{weight} of the edge $e$.
The pair $(G,w)$ is called a \emph{network}.
Most of the time no weights will be specified, and then
it is assumed that $w(e) = 1$ for all $e \in E$. 
The weight of a spanning tree $t \in \cT_G$ is defined 
by $w(t) := \prod_{e \in t} w(e)$. 
We extend the definition of $\mu_G$ to networks
by requiring that each element of $\cT_G$ receives
probability proportional to its weight.

A \emph{network random walk} on $(G,w)$ is a 
Markov chain $\{ S(n) \}_{n \ge 0}$ with state space 
$V$ and transition probabilities:
\eqnst
{ \P [ S(k+1) = v \,|\, S(k) = u ] 
  = \frac{w(u,v)}{\sum_{v' \sim u} w(u,v')}. }
When the weights are constant, we call this 
\emph{simple random walk} on $G$. The definition
of network random walk immediately extends to 
infinite networks as long as for each vertex
$u \in V$ we have $\sum_{v' \sim u} w(u,v') < \infty$.

Let $v_1, \dots, v_N$ be an enumeration of $V$, 
and let $r$ be a fixed vertex of $G$. 
Let $\{ S^j_k \}_{k \ge 0}$, $1 \le j \le N$
be independent network random walks on $G$, with 
$S^j(0) = v_j$. We define a sequence of subtrees 
$\cF_0 \subset \cF_1 \subset \dots \subset \cF_N$ of $G$.
Put $\cF_0 = \{ r \}$, and inductively define for
$j \ge 1$:
\eqnspl{e:Wilson}
{ T_j 
  &:= \inf \{ k \ge 0 : S^j(k) \in \cF_{j-1} \} \\
  \cF_j
  &:= \cF_{j-1} \cup \LE(S^j[0,T^j]). }
It is clear from the construction that $\cF_N$ is a spanning
tree of $G$.

\begin{theorem}[Wilson \cite{Wilson}]
\label{thm:Wilson}
On any finite network, regardless of what enumeration 
was chosen, $\cF_N$ is distributed according to $\mu_G$.
\end{theorem}

Suppose now that $G = (V,E)$ is a locally finite 
infinite recurrent graph. Essentially the same 
method can be applied as in the finite case.
Let $v_1, v_2, \dots$ be an enumeration of $V$,
and let $r \in V$ be fixed. Define $\cF_j$, $j \ge 0$
as in the finite case, and set $\cF := \cup_{j \ge 0} \cF_j$.
Then $\cF$ is a.s.~a spanning tree of $G$

\begin{theorem}[{BLPS \cite[Theorem 5.6]{BLPS01}}]
\label{thm:Wilsonrec}
On any recurrent infinite graph,
regardless of the enumeration chosen, $\cF$ is
distributed according to $\mu$. 
\end{theorem}

Suppose now that $G = (V,E)$ is a locally finite 
infinite transient graph.
Wilson's method can be applied to this case as well,
by letting the root $r$ be ``at infinity''. That is,
let $v_1, v_2, \dots$ be an enumeration of $V$, 
set $\cF_0 := \es$, and define $T_j$ and $\cF_j$ 
as in \eqref{e:Wilson}. Now some of the $T_j$'s will
be infinite, but as noted earlier, loop-erasure
still makes sense due to transience. 
We set $\cF := \cup_{j \ge 1} \cF_j$.
Then $\cF$ is a.s.~a spanning forest of $G$.

\begin{theorem}[{BLPS \cite[Theorem 5.1]{BLPS01}}]
\label{thm:Wilsontrans}
On any transient infinite graph,
regardless of the enumeration chosen, $\cF$ is 
distributed according to $\mu$.
\end{theorem}

\section{Infinite volume limits --- single tree}
\label{sec:limitsI}

Let $G = (V,E)$ be an infinite locally finite graph
as in Section \ref{sec:WSF}. 
Let $V_1 \subset V_2 \subset \dots \subset V$ be an
exhaustion, and recall the wired graph $G_n = (V_n^+,E_n)$. 
In this section we assume that the WSF of $G$
is $\mu$-a.s.~connected and has one end. 
By Theorem \ref{thm:Pemalt}(i), this includes 
$\Z^d$ with $2 \le d \le 4$. The theorem below 
was proved in \cite{AJ04}. There it was stated
in the case of $\Z^d$, $2 \le d \le 4$, however, 
the proof there directly applies to the more general setting.
Nevertheless, below we present a somewhat different proof,
based on coupling. Let us write 
\eqnst
{ \Om_G 
  := \prod_{x \in V} \{ 0, \dots, \deg_G(x) - 1 \}. }

\begin{theorem}[Athreya, J\'arai \cite{AJ04}]
\label{thm:AJlow}
Let $G = (V,E)$ be an infinite locally finite graph.
Suppose that the WSF of $G$ is $\mu$-a.s.~connected
and has one end. There exists a measure $\nu$ on $\Om_G$
such that $\nu_{G_n} \Rightarrow \nu$, independently 
of the exhaustion.
\end{theorem}

Before proving Theorem \ref{thm:AJlow}, let us comment 
on when the assumptions are satisfied.
If $G = (V,E)$ is any graph, an \emph{automorphism} of 
$G$ is a bijection $\varphi : V \to V$, such that 
$\{ x, y \} \in E$ if and only if 
$\{ \varphi(x), \varphi(y) \} \in E$.
We say that $G$ is \emph{vertex-transitive}, 
if for any $x, y \in V$ there exists an automorphism
that takes $x$ to $y$.

Suppose that $G = (V,E)$ is a locally finite 
(vertex)-transitive graph. Let $o \in V$ be a 
fixed vertex of $G$, and let 
$v_n$ be the number of vertices of $G$ with distance
at most $n$ from $o$. It was shown by Lyons, Peres and 
Schramm \cite[Corollary 5.3]{LPS03} and 
BLPS \cite[Corollary 9.6]{BLPS01} that if $v_n \le c n^4$,
then the WSF is a.s.~connected. Regarding the number
of ends, it was shown by BLPS \cite[Theorem 10.3]{BLPS01}
that if $G$ is transitive and transient and the WSF has
a single tree a.s., then that tree has one end a.s. This
was further generalized by Lyons, Morris and Schramm
\cite[Theorem 7.1]{LMS08} who gave a sufficient condition 
in terms of the isoperimetric profile of the graph,
without assuming transitivity. Regarding the recurrent case, 
it was shown in \cite[Theorem 10.6, Proposition 10.10]{BLPS01} 
that in a recurrent transitive graph $G$, the WSF has
one end a.s.~unless $G$ is roughly isometric to $\Z$.

\begin{proof}[Proof of Theorem \ref{thm:AJlow}]
Fix $x \in V$. We will use a subscript $n$ for
objects associated with the graph $G_n$. In particular,
we write $F_{n,x}$, $v_{n,x}$, etc.~for 
the data associated to a $t_n \in \cT_{G_n}$ 
appearing in Lemma \ref{lem:depend}.

Write $\cT_G$ for the set of spanning
trees of $G$ with one end, and let $t \in \cT_G$. 
Due to the one end property, we can think of each 
edge of $t$ being directed towards infinity.
For $u, v \in V$ let us write $u \preceq v$ if there
is a directed path (possibly of length $0$) from 
$u$ to $v$ in $t$, and write $u \prec v$ in the
case when $u \not= v$.
For $y \in \cN_x$ let $\pi_y$ denote the unique 
infinite directed path in $t$ starting at $y$. 
Let $v_x \in V$ be the unique vertex such that
$v_x \in \pi_y$ for all $y \in \cN_x$, and $v_x$ 
is minimal with respect to the relation $\preceq$
(the ``first meeting point''). 
Let us write $\vec{F}_x$ for the following directed 
subtree of $t$:
\eqnst
{ \vec{F}_x 
  := \left\{ e \in t : 
     e_- \in \cup_{y \in \cN_x} \pi_y,\, 
     e_+ \preceq v_x) \right\}. }
Each edge in $\vec{F}_x$ is directed towards $v_x$,
so specifying $\vec{F}_x$ is equivalent to specifying
the undirected, rooted tree $(F_x,v_x)$.

We now define a mapping $\phi_G : \cT_G \to \Om_G$. Let 
\eqnsplst
{ n_x
  &= | \{ y : y \sim x,\, d_{F_x}(v_x,y) < d_{F_x}(v_x,x) \} |, \\
  P_x
  &= \{ e \in \vec{E} : 
     e_- = x,\, d_{F_x}(v_x,e_+) = d_{F_x}(v_x,x) - 1 \}, \\ 
  K_x
  &= \{ \deg_G(x) - n_x, \dots, \deg_G(x) - n_x + |P_x|-1 \}. }
Let $e_x$ be the unique edge of $t$ satisfying $(e_x)_- = x$. 
Set $\eta_x := \alpha_{P_{x}, K_{x}} (e_x)$, $x \in V$.
Define $\phi_G(t) := \eta$, and let $\nu$ be the image
of $\mu$ under the map $\phi_G$.

We show that $\nu_n \Rightarrow \nu$. In fact, we 
consider a coupling between the measures $\mu_n$ and 
$\mu$, with the following property. 
With $\eta_n = \phi_{G_n}(t_n)$ and 
$\eta = \phi_G(t)$, for all finite $A \subset V$ we have 
\eqn{e:agree}
{ \lim_{n \to \infty} \P [ \eta_{n,x} = \eta_x,\, x \in A ] 
  = 1. }
This clearly implies weak convergence.

We first consider the case when $G$ is recurrent.
For any $x \in V$ let
\eqnsplst
{ D_x
  &:= \{ e \in E : \text{$\exists u \in V$ incident
      with $e$ such that $u \preceq v_x$} \}. }  
Due to the assumption on one end, $D_x$ is 
$\mu$-a.s.~finite. 

\begin{lemma} 
\label{lem:determine}
Let $K \subset E$ be a fixed finite set of edges.
On the event $D_x \subset K$, the value of
$(F_x,v_x)$, and hence of $\eta_x = (\phi_G(t))_x$
is determined by the status of the edges in $K$,
that is by the pair $(K \cap t, K \setminus t)$.
Similarly, on the event $D_{n,x} \subset K$, the
value of $(F_{n,x},v_{n,x})$, and hence the value 
of $\eta_{n,x} = (\phi_{G_n}(t_n))_x$ is determined
by the status of the edges in $K$.
\end{lemma}

\begin{proof}[Proof of Lemma \ref{lem:determine}]
All edges of $F_x$ belong to $D_x \cap t$ and 
hence to $K \cap t$. Therefore, $F_x$ is determined 
as the smallest connected set of edges in $K \cap t$ 
containing all vertices of $\cN_x$. It remains to show 
that $v_x$ is also determined. 

We claim that $v_x$ is
the unique vertex belonging to $F_x$ such that there 
exists a path in $K \cap t$ from $v_x$ to the vertex 
boundary of $K$ that is edge-disjoint from $F_x$.
First note that $v_x$ satisfies the requirement, by
virtue of the path in $t$ from $v_x$ to infinity.
Suppose $v \not= v_x$ was another such vertex, and let
$f_1, \dots, f_L \in K \cap t$ be a path from $v$
to the vertex boundary of $K$ that is disjoint from $F_x$. 
By the definition of $D_x$ and induction, we have 
$f_j \in D_x \cap t$, $j = 1, \dots, L$. Let $f \not\in K$ 
be an edge that shares an endvertex with $f_L$. 
The common endvertex of $f_L$ and $f$, call it $u$, 
satisfies $u \prec v_x$. Hence we get $f \in D_x \subset K$, 
a contradiction.
\end{proof}

We continue the proof of \eqref{e:agree} (in the case
whan $G$ is recurrent). Fix $\eps > 0$.
Choose $B \subset E$ a large enough finite set, so that 
\eqn{e:largeB}
{ \P [ \cup_{x \in A} D_x \subset B ] 
  > 1 - \eps. }
Assume $n$ is large enough so that $E_n \supset B$.
The following coupling between $\mu_n$ and $\mu$ is
due to BLPS \cite[Proposition 5.6]{BLPS01}.
Let $u_1, \dots, u_K$ be an enumeration of all vertices
incident with the edges in $B$. We use Wilson's method to 
generate samples $t_n$ (resp.~$t$) from $\mu_n$ (resp.~$\mu$),
where the enumeration of vertices starts with $u_1, \dots, u_K$,
and the root is some fixed vertex $r \in A$. The same random 
walks are used in the case of $G_n$ and $G$, up to the first 
time $\tau_n^j$ when the walk crosses an edge between $V_n$ 
and the sink $s$. After time $\tau_n^j$, the construction on $G_n$ is 
continued using an independent simple random walk on $G_n$ 
started at $s$. Due to recurrence, for large enough $n$,
\eqn{e:noexit}
{ \P [ \text{$\tau_n^j > T^j$ for $j = 1, \dots, K$} ] 
  > 1 - \eps. }
If the event in \eqref{e:noexit} occurs, the status of all edges
in $B$ are the same for $t_n$ and $t$. When the event
in \eqref{e:largeB} also occurs, Lemma \ref{lem:determine},
Lemma \ref{lem:depend}(i), and the definitions of $\phi_G$ 
and $\phi_{G_n}$ imply that $(\phi_G(t))_x = (\phi_{G_n}(t_n))_x$ for 
all $x \in A$. Since $\eps$ was arbitrary, this 
proves \eqref{e:agree} in the recurrent case.

When $G$ is transient, the proof is fairly similar, and
somewhat simpler. This time, we let $u_1, \dots, u_K$ be
an enumeration of $\cup_{x \in A} \cN_x$. On $G$ we use Wilson's 
method rooted at infinity, and on $G_n$ we use it with root
equal to the sink $s$. The constructions use the same 
random walks $S^j$, up to the first exit time $\tau_n^j$
from $V_n$ for $j = 1, \dots, K$. Given $\eps > 0$, 
let $C \subset E$ be a large enough finite set so that 
\eqn{e:contained}
{ \P [ \cup_{x \in A} F_x \subset C ] 
  > 1 - \eps. }
Let $\hat{\tau}^j_C$ be the time of the last visit to $C$
by $S^j$. Due to transience, $\hat{\tau}^j_C < \infty$ a.s.
It follows, using transience again, that if $n$ is large 
enough
\eqn{e:notback}
{ \P \left[ S^j[\tau^j_n, \infty) \cap S^j[0,\hat{\tau}^j_C] 
     = \es,\, j = 1, \dots, K \right]
  > 1 - \eps. }
Note that on the event in \eqref{e:notback}, using the
notation from Section \ref{sec:Wilson}, we have
$\cF_K \cap C = \cF_K^n \cap C$. When the event 
in \eqref{e:contained} also occurs, we have 
$(F_x,v_x) = (F_{n,x}, v_{n,x})$, $x \in A$. This implies, due to 
Lemma \ref{lem:depend}(i) and the definitions of
$\phi_G(t)$ and $\phi_{G_n}(t_n)$,
that $\eta_x = \eta_{n,x}$, $x \in A$. Since $\eps$ was
arbitrary, we obtain \eqref{e:agree} in the transient
case.
\end{proof}

Uniqueness of the limit in Theorem \ref{thm:AJlow} can fail,
if the assumption on one end is dropped. The following theorem
was proved in \cite{JL07}. Let $G_0$ be a connected finite
graph. Let $G$ be the product 
$\Z \times G_0$, that is, $(n_1, u_1)$ and $(n_2, u_2)$
are connected by an edge, if either $n_1 = n_2$ and 
$u_1 \sim u_2$ in $G_0$, or if $u_1 = u_2$ and 
$|n_1 - n_2| = 1$. Write $G_{n,m}$ for the wired graph
associated to $\{n, n+1, \dots, m-1, m\} \times G_0$.

\begin{theorem}[J\'arai, Lyons \cite{JL07}]
\label{thm:JL}
If $G_0$ has at least two vertices, then 
$\{ \nu_{G_{n,m}} : n < 0,\, m > 0 \}$ has precisely
two ergodic weak limit points. 
\end{theorem}

\begin{remark}
Here the WSF on $G$ has two ends a.s., as can be seen by
using Wilson's method, Theorem \ref{thm:Wilsonrec}. 
Therefore, the conditions of Theorem \ref{thm:AJlow} are 
not satisfied. It is a natural question whether this is
the only thing that can go wrong with the existence of
a unique limit. If the answer is yes, this would solve
Open question \ref{open:general limit}. 
\end{remark}

\section{Infinite volume limits --- multiple trees}
\label{sec:limitsII}

\subsection{Statement of result}
\label{ssec:result}

In this section we will be interested in graphs
where the WSF is not a single tree. Theorem \ref{thm:Pemalt}(ii) 
states that this is the case when $G$ is the $\Z^d$ lattice 
for $d \ge 5$. The method of proof of Theorem \ref{thm:AJlow} 
breaks down in this case, because with probability
bounded away from $0$, $v_{n,x}$ equals the sink, and hence
$\{ (F_{n,x},v_{n,x}) \}_{n \ge 1}$ is not tight. 
The following theorem was proved in \cite{AJ04} 
in the case when the exhaustion satisfies a regularity property. 
The restriction on the exhaustion was removed 
in \cite[Appendix]{JR08}, using the result of \cite{JK04}. 

\begin{theorem}[{Athreya, J\'arai \cite{AJ04}; J\'arai, Redig \cite{JR08}}]
\label{thm:AJhigh}
Consider the $\Z^d$ lattice with $d \ge 5$, and let 
$V_1 \subset V_2 \subset \dots \subset \Z^d$ be any exhaustion.
There exists a measure $\nu$ on $\Om_{\Z^d}$ such that 
$\nu_n \Rightarrow \nu$, independently of the exhaustion.
\end{theorem}

The goal of this section is to generalize Theorem \ref{thm:AJhigh}
to other graphs under certain conditions.

Let $G = (V,E)$ be an infinite locally finite graph. 
We denote by $\AUT(G)$ the group of
graph automorphisms of $G$. With the topology of 
pointwise convergence, $\AUT(G)$ is a locally compact
group \cite[Lemma 1.27]{Woess}. 

A function $h : V \to \R$ is called \emph{harmonic}, 
if for every $x \in V$ we have
\eqnst
{ \frac{1}{\deg_G(x)} \sum_{y : y \sim x} h(y) 
  = h(x). }

We make the following assumptions on $G$.

\begin{assumption}
\label{ass:Gcond} \ \\
(i) $G$ is vertex-transitive. \\
(ii) The probability that two independent simple random walks
on $G$ started at some vertex intersect infinitely often
is $0$.\\
(iii) Each component of the WSF of $G$ has one end a.s.\\
(iv) Every bounded harmonic function on $G$ is constant.
\end{assumption}

We are going to prove the following theorem.

\begin{theorem}
\label{thm:trans} 
Let $G = (V,E)$ be an infinite, locally finite graph,
satisfying Assumption \ref{ass:Gcond}(i)--(iv).
There exists a measure $\nu$ on $\Om_G$ such that for any 
exhaustion $V_1 \subset V_2 \subset \dots \subset V$
we have $\nu_{G_n} \Rightarrow \nu$. 
\end{theorem}

Before setting out to prove Theorem \ref{thm:trans}, 
let us discuss examples where the conditions are satisfied.

\medbreak

\emph{Condition (i).} Suppose that $\Gamma$ is a finitely 
generated group, and let $S$ be a fixed finite generating 
set with the property that if $s \in S$ then also 
$s^{-1} \in S$. The (right-)Cayley graph of $(\Gamma,S)$ is the graph 
with vertex set $V = \Gamma$ and edge set   
\eqnst
{ E 
  := \{ \{ x, xs \} : x \in \Gamma,\, s \in S \}. } 
Any Cayley graph is vertex-transitive, as shown by
left-multiplication by elements of $\Gamma$.

\medbreak

\emph{Condition (ii).} Suppose that $G$ is a 
vertex-transitive graph, and let $o$ be a fixed vertex 
of $G$. Write $d(\cdot, \cdot)$ for graph distance in $G$. 
Let 
\eqnst
{ V(n)
  := \left| \{ x \in V : d(o,x) \le n \} \right|. }
Suppose that there exists a constant $c > 0$ such that
$V(n) \ge c n^5$. Let $\{ S_n \}_{n \ge 0}$ be simple
random walk on $G$. Due to \cite[Corollary 14.5]{Woess}, 
the return probability of $S$ satisfies 
$\P [ S_{2n} = o \,|\, S_0 = o ] \le C n^{-5/2}$. 
As explained in \cite[Section 5]{LPS03}, this implies 
that the expected number of intersections (with multiplicity)
between two independent simple random walks starting at
$o$ is finite. Hence (ii) is satisfied in this case.
Note that by \cite[Theorem 9.4]{BLPS01}, the WSF has 
infinitely many trees a.s., whenever (i) and (ii) are
satisfied.

\medbreak

\emph{Condition (iii).} Suppose that $G$ is a 
vertex-transitive graph satisfying $V(n) \ge c n^3$. 
It follows from results of Lyons, Morris and Schramm
\cite[Theorem 7.1]{LMS08}, \cite[Corollary 7.3]{LMS08},
that every tree of the WSF has one end a.s. In the cases
when the WSF is a single tree, and when the WSF is disconnected
with $\AUT(G)$ unimodular, this was earlier 
proved by BLPS \cite[Theorem 10.3]{BLPS01}, 
\cite[Theorem 10.4]{BLPS01}.
Hence (iii) is satisfied for a large class of graphs.

\medbreak

\emph{Condition (iv).} Let $G = (V,E)$ be a graph on which
the group $\Gamma \subseteq \AUT(G)$ acts transitively, 
i.e., for any $x, y \in V$ there exists 
$\varphi \in \Gamma$ such that $\varphi(x) = y$. 
Examples where Assumption \ref{ass:Gcond}(iv) 
is satisfied are given by \emph{nilpotent} groups $\Gamma$.
Recall that for $a, b \in \Gamma$, their commutator
is defined as $[a,b] := a^{-1} b^{-1} a b$. Let 
$\Gamma_1 := \Gamma$, and for $k \ge 2$ let
$\Gamma_k$ be the subgroup of $\Gamma$ generated
by all elements of the form 
$[\dots[[a_1,a_2],a_3],\dots,a_k]$. Then 
\eqnst
{ \Gamma
  = \Gamma_1 \supseteq \Gamma_2 \supseteq \Gamma_3 \supseteq \dots }
is called the \emph{lower central series} of $\Gamma$.
If there exists an $r$ such that $\Gamma_{r+1}$ is the 
trivial group, $\Gamma$ is called 
\emph{nilpotent} \cite[Chapter 10]{Hall}.
It was shown in \cite{DM61} that if $\Gamma$ is 
nilpotent then any bounded harmonic function on $G$
is constant.

\begin{remark}
We believe that the technical Assumption \ref{ass:Gcond}(iv)
is not necessary. However, at present the only example
where we know the existence of $\nu$ without this assumption
is the case of a regular tree, discussed in 
Section \ref{sec:limitsIII}.
\end{remark}

\subsection{Notation and coupling}
\label{ssec:notation}

We prepare for the proof of Theorem \ref{thm:trans} 
by defining the appropriate analogue of $(F_x,v_x)$. 
This is done in the same way as for the case 
of $\Z^d$, $d \ge 5$ in \cite{AJ04}. In order to be 
self-contained, we give the details.
The idea behind the definitions is that there is probability
bounded away from zero, as $n \to \infty$, 
that two given vertices $y_1, y_2 \in \cN_x$ will be connected 
through the sink $s$. This means that $(F_{n,x},v_{n,x})_{n \ge 1}$
is not tight. We want to replace it with an object that 
is tight, by removing the connections through $s$.

We first give the finite volume definitions. We use notation 
similar to Section \ref{sec:limitsI}, that is, lower
indices $n$ refer to the graph $G_n$. Fix $x \in V$, 
and assume that $\cN_x \subset V_n$. Let $t_n \in \cT_{G_n}$, 
and define the forest $\bt_n := t_n \setminus \{ s \}$. Let
\eqnspl{e:subtrees3}
{ K_{n,x} (\bt_n)
  &:= \text{number of connected components of $\bt_n$ 
     intersecting $\cN_x$} \\
  t^{(1)}_{n,x}, \dots, t^{(K_{n,x})}
  &:= \text{the components of $\bt_n$ that intersect $\cN_x$} \\
  A^{(i)}_{n,x}
  &:= t^{(i)}_{n,x} \cap \cN_x, \quad 1 \le i \le K_{n,x}. }
Here the indexing of the $t^{(i)}$'s and the $A^{(i)}$'s 
is determined as follows. We fix an ordering of $\cN_x$, let us say 
$\cN_x = \{ y_0 = x, y_1, \dots, y_{\deg_G(x)} \}$.
We let $t^{(1)}_{n,x}$ be the component of $\bt_n$ containing $y_0$, 
let $t^{(2)}_{n,x}$ be the component containing the earliest 
$y_i$ not in $t^{(1)}_{n,x}$, etc. 

Let $v^{(i)}_{n,x}$ be the unique vertex $v$ of $t^{(i)}_{n,x}$
such that $v \in \pi_{n,y}$ for all 
$y \in A^{(i)}_{n,x}$, and $d_t(s,v)$ is maximal. 
We define
\eqn{e:F(i)s}
{ \vec{F}^{(i)}_{n,x} 
  := \left\{ e \in t^{(i)}_{n,x} : 
     e_- \in \bigcup_{y \in A^{(i)}_{n,x}} \pi_{n,y},\, 
     d_t(s,e_+) \ge d_t(s,v^{(i)}_{n,x}) \right\}. }
Specifying the $\vec{F}^{(i)}_{n,x}$'s is equivalent to specifying
the undirected rooted trees $(F^{(i)}_{n,x},v^{(i)}_{n,x})_{i=1}^{K_{n,x}}$.

We introduce the set of relative distances:
\eqnst
{ d^{(i,j)}_{n,x}
  = d^{i,j}_{n,x} (t_n)
  = d_n(v^{(i)}_{n,x}, s) - d_n(v^{(j)}_{n,x}, s),
    \quad 1 \le i < j \le K_{n,x}. }
Due to Lemma \ref{lem:depend}(ii),
$\eta_{n,x} = (\phi_n(t_n))_x$ only depends on the data:
\eqnst
{ K_{n,x}(t_n), \quad
  ( F^{(i)}_{n,x}(t_n), v^{(i)}_{n,x}(t_n) )_{i=1}^{K_{n,x}}, \quad
  \{ d^{(i,j)}_{n,x}(t_n) \}_{1 \le i < j \le K_{n,x}}. }
When each tree in the WSF on $G$ has one end, we can expect
that the joint law of 
\eqnst
{ K_{n,x}, \quad (F^{(i)}_{n,x}, v^{(i)}_{n,x})_{i=1}^{K_{n,x}} } 
converges as $n \to \infty$. The candidate for the limit 
is given by the natural analogues in the graph $G$, 
that we now define. 

Let $\cT_G \subset \{0,1\}^E$ denote the set of all
spanning forests of $G$ such that each component is
infinite and has one end. Let $\bt \in \cT_G$. Due to the 
one end property, we can direct each edge of $\bt$ towards
the end of the component containing it. Again, we write
$u \preceq v$, if there is a directed path from $u$ to $v$. 
As in Section \ref{sec:limitsI}, for $y \in V$ we denote 
by $\pi_y$ the unique infinite directed path in $\bt$ 
starting at $y$. Fix $x \in V$, and let
\eqnspl{e:subtrees2}
{ K_{x} (\bt)
  &:= \text{number of connected components of $\bt$ 
     intersecting $\cN_x$} \\
  t^{(1)}_x, \dots, t^{(K_x)}_x
  &:= \text{the components of $\bt$ that intersect $\cN_x$} \\
  A^{(i)}_x 
  &:= t^{(i)}_x \cap \cN_x, \quad 1 \le i \le K_x. }
Here the indexing of the $t^{(i)}$'s and $A^{(i)}$'s follows 
the same rule as in the case of $G_n$. 
Let $v^{(i)}_x$ be a vertex of $t^{(i)}_x$ minimal with 
respect to the relation $\preceq$ among all vertices $v$ 
with the property that $y \preceq v$ for all 
$y \in A^{(i)}_x$. Such a vertex exists,
due to the one end property, and there is a unique
minimal one. Let
\eqnst
{ \vec{F}^{(i)}_x 
  := \left\{ e \in t^{(i)}_x : 
     e_- \in \bigcup_{y \in A^{(i)}_x} \pi_y,\, 
     e_+ \preceq v^{(i)}_x) \right\}, \quad 1 \le i \le K_x. }
Specifying the $\vec{F}^{(i)}_x$'s is equivalent to specifying
the undirected, rooted trees $(F^{(i)}_x,v^{(i)}_x)$.

\begin{lemma}
\label{lem:highDcoupling}
Suppose that $G = (V,E)$ is a transient
graph that satisfies Assumption \ref{ass:Gcond}(ii)--(iii).
For any finite $A \subset V$ there is a coupling of 
$\mu_n$, $n \ge 1$, and $\mu$ such that in this coupling 
\eqn{e:coupled}
{ \lim_{n \to \infty} \P \left[ K_{n,x} = K_x,\, 
    (F^{(i)}_{n,x}, v^{(i)}_{n,x}) 
    = (F^{(i)}_x,v^{(i)}_x),\, 
    1 \le i \le K_{n,x} \right]
  = 1. }
\end{lemma} 

\begin{proof}
Let $u_1, \dots, u_L$ be an enumeration of 
$\cup_{x \in A} \cN_x$. On $G$ we use Wilson's 
method rooted at infinity with random walks $S^j$,
started at $u_j$ for $j = 1, \dots, L$. 
On $G_n$, we use Wilson's method 
with root equal to the sink, and with the same random 
walks $S^j$, up to their first exit time $\tau_n^j$ 
from $V_n$. Recall the notation from 
Section \ref{sec:Wilson}: $(\cF_i)_{i \ge 0}$ 
and $(\cF_{n,i})_{i \ge 0}$ are the growing forests 
constructed by Wilson's method, and $T^j$ is the hitting time 
of $\cF_{j-1}$ by $S^j$. For any $C \subset E$, let 
\eqnst
{ \hat{\tau}^j_C
  := \sup \{ k \ge 0 : S^j(k) \in C \}. }
Let $J \subset \{ 1, \dots, L \}$ be the (random) set of 
indices such that $T^j = \infty$. 

Given $\eps > 0$, let $C \subset E$ be a large enough 
finite set such that
\eqn{e:contained-hits}
{ \P \left[ \cup_{j \not\in J} S^j[0,T^j] \subset C \right]
  > 1 - \eps. }
By transience, we can find $n_1$, such that 
for all $n \ge n_1$ we have
\eqn{e:contained-paths}
{ \P \left[ \text{for all $j \in J$ we have 
     $S^j[\tau^j_n, \infty) 
     \cap S^j[0,\hat{\tau}^j_{C}) = \es$} \right]
  > 1 - \eps. }
The significance of the event in \eqref{e:contained-paths}
is that on this event, the loop-erasing procedure
on $S^j[0,\infty)$ after time $\tau^j_n$ has no effect 
on the configuration in $C$, so the configurations in $C$
will be the same when the algorithm is run in $G_n$ and $G$.

Observe that if $i, j \in J$, $i < j$, then 
$\LE(S^i[0,\infty)) \cap S^j[0,\infty) = \es$. 
Assumption \ref{ass:Gcond}(ii) and transitivity implies
that almost surely
$|S^j[0,\infty) \cap S^i[0,\infty)| < \infty$. Since the 
points in this intersection are not present in
$\LE(S^i[0,\infty))$, we can find $n_2$ large enough such
that for all $n \ge n_2$ we have
\eqn{e:contained-infinite-paths}
{ \P \left[ \text{for all $i,j \in J$, $i < j$ we have 
     $\LE(S^i[0,\tau^i_n)) \cap S^j[0,\tau^j_n) 
     = \es$} \right]
  > 1 - \eps. }

Assume now the intersection of the events in 
\eqref{e:contained-hits}, \eqref{e:contained-paths} and
\eqref{e:contained-infinite-paths}. 
Let $n \ge \max \{ n_1, n_2 \}$. We prove that the 
event in \eqref{e:coupled} then must occur, 
implying the Lemma. We show that for all 
$1 \le i \le L$ the following holds:
\begin{itemize}
\item[(i)] if $T^i = \infty$ then $T^i_n = \tau^i_n$
and $\LE(S^i[0,\infty))$ and $\LE(S^i[0,\tau^i_n]$
agree up to their last visit to $C$;
\item[(ii)] if $T^i < \infty$ then $T^i = T^i_n < \tau^i_n$
and $\LE(S^i[0,T^i]) = \LE(S^i[0,T^i_n]) \subset C$;
\item[(iii)] $\cF_{n,i} \cap C = \cF_i \cap C$.
\end{itemize}

The proof is by induction on $i$. For $i = 1$ we 
have $T^1 = \infty$ and $T^1_n = \tau^i_n$ always,
so we are in case (i). Let $\gamma$ be the 
initial segment of $\LE(S^1[0,\tau^1_n])$ up to 
the last visit to $C$. Due to the event in 
\eqref{e:contained-paths}, $S^1$ makes no further
visit to $\gamma$ after time $\tau^1_n$, and 
therefore the initial segment of $\LE(S^1[0,\infty))$
up to the last exit from $C$ coincides with $\gamma$,
as required. 

Consider now $2 \le i \le L$. We prove (i).
The event in \eqref{e:contained-infinite-paths} implies
that $S^i[0,\tau^i_n)$ does not intersect any of the 
paths $\LE(S^j[0,\tau^j_n))$ with $j < i$, $j \in J$.
It also does not intersect $\LE(S^j[0,T^j_n])$ for
$j < i$, $j \not\in J$, since by the induction 
hypothesis for $j$, case (ii), we have
$\LE(S^j[0,T^j_n]) = \LE(S^j[0,T^j])$ and $T^i = \infty$.

We now prove (ii). 
By virtue of the event in \eqref{e:contained-hits},
$S^i[0,T^i]$ does not leave $C$. By the induction 
hypothesis (iii) we have $\cF_{n,i-1} \cap C = \cF_{i-1} \cap C$,
and the claim in (ii) follows immediately.

Statement (iii) follows from (i) and (ii).

It follows immediately from (i) and (ii) that 
that $K_{n,x} = K_x$. It also follows from (i) and (ii)
that for each $x \in A$ and $1 \le i \le K_x$ we have 
$(F^{(i)}_{n,x}, v^{(i)}_{n,x}) = (F^{(i)}_x, v^{(i)}_x) \subset C$.
This completes the proof.
\end{proof}

\subsection{Permutation of components}
\label{ssec:perm}

As in \cite{AJ04}, the key difficulty to overcome is 
to analyze the behaviour of the $d^{(i,j)}_{n,x}$'s. 
Lemma \ref{lem:depend}(ii) implies that 
when $\left| d^{(i,j)}_{n,x} \right|$ is large for all 
$1 \le i < j \le K_{n,x}$ then their exact value is 
irrelevant. More precisely, if 
\eqn{e:diam-large}
{ \min \left\{ \left| d^{(i,j)}_{n,x} \right| 
     : 1 \le i < j \le K_{n,x} \right\}
  > \max \{ \diam(F^{(i)}_{n,x}) : 1 \le i \le K_{n,x} \}, } 
then all that matter for the value of 
$\eta_{n,x} = \varphi_{G_n}(t_n)_x$ are the signs of 
$d^{(i,j)}_{n,x}$. Therefore, we introduce the 
permutation $\sigma_{n,x}$ of $\{ 1, \dots, K_{n,x} \}$
by requiring
\eqnst
{ d_{t_n} (v^{(\sigma(1))}_{n,x},s) 
  \le d_{t_n} (v^{(\sigma(2))}_{n,x},s)
  \le \dots 
  \le d_{t_n} (v^{(\sigma(K_{n,x}))}_{n,x},s). }
In case of ties, we break them according to an arbitrary 
fixed rule. We write $\Sigma_k$ for the set of
permutations of $\{ 1, \dots, k \}$, so that 
on the event $\{ K_{n,x} = k \}$, we have 
$\sigma_{n,x} \in \Sigma_k$.
We summarize the above observations on the
dependence on $\sigma_{n,x}$ in the following lemma.

\begin{lemma}
\label{lem:only-perm}
Let $G = (V,E)$ be an infinite graph.
For every $x \in V$, $1 \le k \le \deg_G(x)$ and 
$V_n \supset \cN_x$ there exist functions 
$f_{k,x} (F^{(1)},v^{(1)}, \dots, F^{(k)}, v^{(k)}, s)$
(where $s \in \Sigma_k$) with values in 
$\{ 0, \dots, \deg_G(x)-1 \}$ such that
whenever \eqref{e:diam-large} holds, we have
\eqnst
{ \eta_{n,x} 
  = f_{K_{n,x},x} (F^{(1)}_{n,x}, v^{(1)}_{n,x}, \dots, 
    F^{(K_{n,x})}_{n,x}, v^{(K_{n,x})}_{n,x}, \sigma_{n,x} ). }
\end{lemma}

We will show that if all bounded harmonic functions are 
constant, then $\sigma_{n,x}$ is asymptotically uniform. 
More precisely, conditioned on 
$K_{n,x} = k$ and $(F^{(i)}_{n,x}, v^{(i)}_{n,x})_{i=1}^{K_{n,x}}$, 
$\sigma_{n,x}$ converges in distribution, as
$n \to \infty$, to a uniform random element of $\Sigma_k$.
Assuming this (and a certain consistency 
property between the permutations corresponding to 
different $x_1, x_2 \in V$), 
we can define the measure $\nu$ that is the candidate for 
the limit.

Let $\bt$ be a sample from the WSF on $G$. Consider a 
random linear ordering of the components of $\bt$ that has the
property that it induces the uniform permutation on any
finite subset of components. This can be realized for 
example by assigning i.i.d.~$\Unif(0,1)$ variables
to the components, and considering the ranking induced 
by these. Given components $t_1 \not= t_2$ of $\bt$, we write
$t_1 < t_2$, if $t_1$ preceeds $t_2$ in the
ordering. For any $x \in V$, define $\sigma_x \in \Sigma_{K_x}$
by requiring:
\eqnst
{ \sigma_x(i) < \sigma_x(j) \quad \text{if and only if} \quad
  t^{(i)}_x < t^{(j)}_x \qquad \text{for all $1 \le i < j \le K_{x}$}. }
Define the configuration $\eta \in \Omega_G$ by setting
\eqn{e:etax}
{ \eta_x 
  := f_{K_x,x} (F^{(1)}_x, v^{(1)}_x, \dots,  
     F^{(K_x)}_x, v^{(K_x)}_x, \sigma_x), }
that is defined $\mu$-a.s., under Assumption \ref{ass:Gcond}(iii).
Let $\nu$ be the image of the measure $\mu$ under the 
map $\bt \mapsto \eta$.

In order to prove Theorem \ref{thm:trans}, we need to show
that for any $A \subset V$ finite, the joint distribution 
of $\{ \eta_{n,x} \}_{x \in A}$ converges to the joint 
distribution of $\{ \eta_x \}_{x \in A}$. We extend 
to this situation some of the definitions made for single points. 
Namely, let
\eqnspl{e:subtrees4}
{ K_{n,A} (\bt_n)
  &:= \text{number of connected components of $\bt_n$ 
     intersecting $\cup_{x \in A} \cN_x$} \\
  t^{(1)}_{n,A}, \dots, t^{(K_{n,A})}_{n,A}
  &:= \text{the components of $\bt_n$ that 
     intersect $\cup_{x \in A} \cN_x$}. }
We define $(F^{(i)}_{n,A}, v^{(i)}_{n,A})$, $d^{(i,j)}_{n,A}$ 
and $\sigma_{n,A}$ completely analogously to the single point case.
We also introduce $K_A$, $t^{(1)}_A, \dots, t^{(K_A)}_A$,
$(F^{(i)}_A, v^{(i)}_A)$ in the infinite graph $G$. 

The following two propositions make precise the intuition 
about fluctuations of $d^{(i,j)}_{n,A}$ and the uniformity of 
$\sigma_{n,A}$. 

\begin{proposition}
\label{prop:fluct}
Suppose $G = (V,E)$ satisfies Assumption \ref{ass:Gcond}(i)--(iii).
For any finite $A \subset V$ we have
\eqn{e:max}
{ \lim_{M \to \infty} \limsup_{n \to \infty} 
    \P \left[ \min_{1 \le i < j \le K_{n,A}} 
    \left| d^{(i,j)}_{n,A} \right| \le M \right]
  = 0. }
\end{proposition}

\begin{proposition}
\label{prop:unif}
Suppose $G = (V,E)$ satisfies Assumption \ref{ass:Gcond}(i)--(iv).
For any finite $A \subset V$, any $k \ge 1$, $s \in \Sigma_k$ and 
any sequence of finite rooted trees $(F^{(i)}, v^{(i)})_{i=1}^k$ we have
\eqnspl{e:uniperm}
{ &\lim_{n \to \infty} \P \left[ K_{n,A} = k,\, \sigma_{n,A} = s,\, 
     (F^{(i)}_{n,A}, v^{(i)}_{n,A}) = (F^{(i)}, v^{(i)}),\, 
     1 \le i \le k \right] \\
  &\qquad = \frac{1}{k!}
    \P \left[ K_A = k,\, 
     (F^{(i)}_A, v^{(i)}_A) = (F^{(i)}, v^{(i)}),\, 
     1 \le i \le k \right]. }
\end{proposition}

We prove Proposition \ref{prop:fluct} in the next section, 
and Proposition \ref{prop:unif} in Section \ref{ssec:unif}.
The short proof of Theorem \ref{thm:trans}, using 
the two Propositions, is at the end of Section \ref{ssec:unif}.

\subsection{Lower bound on fluctuations}
\label{ssec:fluct}

We start with some preparations for the proof of 
Proposition \ref{prop:fluct}. 
Let $k \ge 1$ and let $v^{(1)}, \dots, v^{(k)} \in V$ be 
fixed vertices. We analyze the event 
\eqn{e:k-event}
{ \{ K_x = k,\, v^{(1)}_x = v^{(1)}, \dots, v^{(k)}_x = v^{(k)} \} }
using Wilson's method. Let 
$u_1, \dots, u_L$ be an enumeration of
$\{ v^{(1)}, \dots, v^{(k)} \} \cup \left( \cup_{x \in A} \cN_x \right)$
such that $u_i = v^{(i)}$ for $i = 1, \dots, k$. Similarly to the
proof of Lemma \ref{lem:highDcoupling}, we couple
the algorithms in $G$ and in $G_n$ by using the same 
random walks in Wilson's method, with the walk $S^j$ starting at
$u_j$. (But note that this time the enumeration is different.) 
On the event \eqref{e:k-event}, for fixed $1 \le i < j \le k$, 
the occurrence of $|d^{(i,j)}_{n,A}| \le M$ implies that the lengths 
of the two (independent) paths 
$\gamma^{i} = \LE( S^{i} [0, \tau^i_n] )$ and
$\gamma^{j} = \LE( S^{j} [0, \tau^j_n] )$ 
differ by at most $M$. This will be unlikely, if there is any
fluctuation in the length of the paths, and we show that
this is the case whenever $G$ is not a tree. The proof will 
be based on some lemmas that follow. The first lemma makes 
a deterministic statement about the existence of a cycle with 
two infinite paths satisfying some requirements. 
The significance of the statement is that it will allow
us to construct two finite random walk paths of equal number 
of steps between two vertices such that the loop-erasures of the
paths have different lengths. This will be sufficient to establish 
non-trivial fluctuations in the length of the loop-erased walk.

In what follows, $o$ will denote a fixed vertex of $G$.

\begin{lemma}
\label{lem:cycle}
Let $G = (V, E)$ be an infinite vertex-transitive graph that 
is not a tree. There exists a cycle $C = \{t_1,\dots,t_L\}$ in $G$, 
vertices $t_I$, $t_J$ in $C$ and infinite paths 
$\pi = \{ \pi_0, \pi_1, \dots \}$ and 
$\rho = \{ \rho_0, \rho_1, \dots \}$ in $G$ such that:\\
(i) $t_I$ and $t_J$ are not antipodal in $C$, i.e.~$I \not= J + L/2$ 
mod $L$;\\
(ii) $\pi(0) = t_I$ and $\rho(0) = t_J$;\\
(iii) $C$, $\pi[1,\infty)$ and $\rho[1,\infty)$ are disjoint. 
\end{lemma}

\begin{proof}
The conditions imply that the vertex degree is $\ge 3$.
Let $C$ be a cycle of length $\ge 3$ in $G$
that passes through $o$. We assume that $C$ has minimal length.
Using transitivity, we see that there exists a bi-infinite path 
$\dots, s_{-2}, s_{-1}, s_0 = o, s_1, s_2, \dots$ in $G$. Let 
\eqnsplst
{ v 
  := s_{-k}, \quad \text{where $k = \max \{ r \ge 1: s_{-r} \in C \}$;} \\
  w 
  := s_l, \quad \text{where $l = \max \{ r \ge 1 : s_r \in C \}$.} } 
If $C$ has odd length, we can set $t_i = v$, $t_j = w$,
$\pi = \{ s_{-k}, s_{-(k+1)}, \dots \}$, $\rho = \{ s_l, s_{l+1}, \dots \}$.
Henceforth assume that $|C| \ge 4$ and $|C|$ even. If $v$ and $w$ are not 
antipodal, there is nothing further to prove. Assume that $v$ and $w$ 
are antipodal, and let $u \in C$ be a neighbour of $v$. 

\emph{Case (a): $|C| = 4$.} Let $u'$ be the other neighbour of $v$ in $C$.
Since $u$ has at least three distinct 
neighbours, $u$ has a neighbour $u_1 \not= v,w$.

If $u_1 = u'$, then $\{v, u, u_1\}$ is a cycle of length $3$, 
contradicting the minimality of $C$.

If $u_1 = s_{-k_1}$ for some $k_1 > k$, then the triple 
$C' := \{ u_1 = s_{-k_1}, s_{-k_1+1}, \dots, s_{-k} = v, u \}$,
$\pi := s [-k_1, -(k_1+1), \dots)$, 
$\rho := \{ u, w = s_l, s_{l+1}, \dots \}$
satisfies the requirements of the Lemma.
Similarly, if $u_1 = s_{l_1}$ for some $l_1 > l$, we are done.

If none of the above holds, select an infinite self-avoiding path 
\eqnst
{ u, u_1, u_2, u_3 \dots; } 
such a path is easily seen to exist, using transitivity. 
If this path is disjoint from 
\eqn{e:B}
{ B
  := C \cup \{ s_{-k}, s_{-(k+1)}, \dots \} 
     \cup \{ s_l, s_{l+1}, \dots \}, } 
then the triple $C$, $\pi = \{ s_{-k}, s_{-(k+1)}, \dots \}$ 
and $\rho = \{ u, u_1, u_2, \dots \}$ satisfies the 
requirements of the Lemma. 
Therefore, suppose that for some $m \ge 1$ we have 
$u_m \in B$, and let $m$ be the smallest index 
with this property.

If $u_m = u'$, then the triple
$C' = \{v, u, u_1, \dots, u_m = u'\}$, 
$\pi = \{ v = s_{-k}, s_{-(k+1)}, \dots \}$,
$\rho = \{ u, w = s_l, s_{l+1}, \dots \}$ satisfies
the requirements.
If $u_m = v$, then necessarily $m \ge 2$ and the triple 
$C' = \{ v, u, u_1, \dots, u_{m-1} \}$,
$\pi = \{ s_{-k}, s_{-(k+1)}, \dots \}$,
$\rho = \{ u, w = s_l, s_{l+1}, \dots \}$ works.
If $u_m = s_{-k_1}$ for some $k_1 > k$, then consider 
the cycle 
$C' = \{ v, u, u_1, \dots, u_{m} = s_{-k_1}, s_{-k_1+1}, \dots, s_{-k-1} \}$.
If $u_m$ and $u$ are not antipodal in $C'$, then we can set
$\pi = \{ s_{-k_1}, s_{-(k_1+1)}, \dots \}$ and
$\rho = \{ u, w = s_l, s_{l+1}, \dots \}$. If they are antipodal, 
then $u_m$ and $v$ are not antipodal in $C'$, so we can
replace $\rho$ by $\rho' = \{ v, u', w = s_l, s_{l+1}, \dots \}$.
Similarly, we are done if $u_m = s_{l_1}$ for some $l_1 \ge l$.
This complete Case (a).

\emph{Case (b): $|C| \ge 6$, $|C|$ even.} Again we start by
letting $u$ be a neighbour of $v$ in $C$, and let $\bar{u}$ be 
the neighbour of $u$ in $C$ different from $v$. 
Let $u_1$ be a neighbour of $u$ different from $v$, $\bar{u}$.
Let $b$ denote the path in $C$ from $u$ to $w$ passing through $\bar{u}$, and
let $a$ denote the path in $C$ from $v$ to $w$ not passing through $u$.

Select a self-avoiding path $u, u_1, u_2, \dots$.
If this path does not intersect $B$ (as defined in \eqref{e:B}), 
we are done similarly to Case (a).
If there is an intersection, let the first one be $u_m$ ($m \ge 1$).

If $u_m = s_{-k_1}$ for some $k_1 > k$, we consider
$C' = \{s_{-k_1}, s_{-k_1+1}, s_{-k} = v, u, u_1, \dots, u_{m-1}\}$.
If $u_m$ and $u$ are not antipodal in $C'$, we set
$\pi = \{ s_{-k_1}, s_{-(k_1+1)}, \dots \}$ and
$\rho = b \cup \{ s_{l+1}, s_{l+2}, \dots \}$. If they are
antipodal, then $u_m$ and $v$ are not antipodal, and we
can replace $\rho$ by $\rho' = a \cup \{ s_{l+1}, s_{l+2}, \dots \}$.
If $u_m = v$ (and necessarily $m \ge 2$), then we set
$C' = \{ v, u, u_1, \dots, u_{m-1}\}$, 
$\pi = \{ s_{-k}, s_{-(k+1)}, \dots \}$ and 
$\rho = b \cup \{ s_{l+1}, s_{l+2}, \dots \}$. 

If $u_m = s_{l_1}$ for some $l_1 > l$, then the triple
$C$, $\pi = \{ v = s_{-k}, s_{-(k+1)}, \dots \}$, 
$\rho = \{ u, u_1, \dots, u_m = s_{l_1}, s_{l_1+1}, \dots \}$
works. 
If $u_m = w$, we note that the path $\{ u, u_1, \dots, u_m = w \}$ 
has to be longer than $b$, otherwise their union gives
a cylce shorter than $C$. In particular in the cycle 
$C' = \{ u, u_1, \dots, u_m \} \cup b$ the vertices $u$ and $w$ 
are not antipodal. Hence the choice
$\pi = \{ u, v = s_{-k}, s_{-(k+1)}, \dots \}$ and 
$\rho = \{ w = s_l, s_{l+1}, \dots \}$ works. 

If $u_m \in a \setminus \{ v, w \}$, we can find a cycle containing
$u, v$, and part of $a$, and use $\pi = \{ v = s_{-k}, s_{-(k+1)}, \dots \}$
and $\rho = b \cup \{ s_{l+1}, s_{l+2}, \dots\}$.
Finally, assume that $u_m \in b \setminus \{ u, w \}$, and let
$c = \{ u, u_1, \dots, u_m \}$. Let $b'$ be the subpath of $b$ from
$u$ to $u_m$. If $|c| = |b'|$, then $c \cup b'$ yields a cycle
shorter than $C$, a contradiction (note that the case $m = 1$ 
is excluded here, since $u_1 \not= \bar{u}$). 
Therefore, $|c| \not= |b'|$, and hence
$u$ and $u_m$ are not antipodal in the cycle $c \cup b'$. Therefore,
we can set $\pi = \{ u, v = s_{-k}, s_{-(k+1)}, \dots \}$ and 
$\rho = (b \setminus b') \cup \{ s_l, s_{l+1}, \dots \}$.
\end{proof}

Consider the configuration constructed in Lemma \ref{lem:cycle}. 
We assume the labeling is such that $I = 1$. 
Shifting by an automorphism we may assume that $\pi(1) = o$.
Let $G_0$ denote the finite graph consisting of the cycle
$\{ t_1, t_2, \dots, t_L \}$ together with the edges 
$\{ \pi(1), t_1 \}$ and $\{ \rho(1), t_J \}$.
We define two nearest neighbour paths in $G_0$ such that:\\
(i) they both start at $\pi(1)$ and end at $\rho(1)$;\\
(ii) they both visit each edge of $G_0$;\\
(iii) they have the same number of steps $2L + J + 1$;\\
(iv) their loop-erasures have different lengths.\\
Let
\eqnsplst
{ \beta_1 
  &:= [\pi(1), t_1, t_2, \dots, t_L, t_1, t_L, t_{L-1}, 
      \dots, t_J, t_{J-1}, \dots, 
      t_2, t_1, t_2, \dots, t_J, \rho(1) ] \\
  \beta_2
  &:= [\pi(1), t_1, t_2, \dots, t_L, t_1, t_L, t_{L-1}, 
      \dots, t_J, 
      \rho(1), t_J, \dots, \rho(1), t_J, \rho(1)]. }
Here $\rho(1), t_J$ is repeated as many times as necessary
so that the length of $\beta_2$ is $2L + J + 1$. The loop-erasure of
$\beta_1$ has length $L - J + 3$, while the loop-erasure of
$\beta_2$ has length $J + 1 \not= L - J + 3$. 

We want to show that a long loop-erased random walk in $G$ 
will contain copies of $\LE(\beta_1)$ and $\LE(\beta_2)$ 
with positive densities. We can do this by an adaptation 
of an argument of Lawler \cite[Theorem 7.7.2]{Lawler}.
For this it will be convenient to 
define a bi-infinite simple random walk by letting
$\{ S(m) \}_{m \ge 0}$ and $\{ S(-m) \}_{m \ge 0}$ be 
independent realizations of simple random walk on $G$
starting at $o$.

We set $M = 2L + J + 1$, and we consider the blocks 
\eqnst
{ B_k 
  = [ S(Mk), S(Mk+1), \dots, S(M(k+1)) ]. } 
Let $\AUT_o$ denote the stabilizer of $o$ in $\AUT(G)$.
By \cite[Lemma (1.27)]{Woess}, $\AUT_o$ is compact, and
hence it carries a right-invariant Haar measure $\la$
of total mass $1$. For each $x \in V$ we fix an automorphism 
$\phi_x$ that takes $o$ to $x$.

\begin{definition}
\label{def:good}
We say that an index $j \ge 0$ is \emph{good}, if the following 
conditions are satisfied:
\begin{itemize}
\item[(a)] For some $\psi \in \AUT_o$ we have 
$\psi \phi_{S(Mj)}^{-1} B_j = \beta_1$ or 
$\psi \phi_{S(Mj)}^{-1} B_j = \beta_2$;
\item[(b)] $S(-\infty, Mj) \cap B_j = \es$ and
  $S(M(j+1), \infty) \cap B_j = \es$;
\item[(c)] $S(-\infty,Mj] \cap S[M(j+1), \infty) = \es$.
\end{itemize}
\end{definition}

Let $b := \P [ \text{$0$ is good} ]$. 
In what follows, we write 
\eqnst
{ B(x,k)
  := \{ y \in V : \dist(x,y) \le k \}. }
We also introduce the notation
$\xi_B = \inf \{ n \ge 0 : S(n) \in B \}$ for the 
hitting time of $B$ by $S$. 
The following lemma shows that good indices occur
with positive frequency.

\begin{lemma}
\label{lem:goodblocks}
Suppose the graph $G = (V,E)$ satisfies 
Assumption \ref{ass:Gcond}(i)--(ii). \\
(1) We have $b > 0$.\\
(2) For any $\eps > 0$, we have
\eqnst
{ \P \left[ \parbox{7.5cm}{$\exists K_0$ 
      $\forall K \ge K_0$
      there are at least $(b-\eps)K$
      good indices among $0, \dots, K-1$} \right] 
  = 1. }  
\end{lemma}

\begin{proof}[Proof of (1).] 
Assumption \ref{ass:Gcond}(ii) implies that $G$ is
transient, in particular. Transience and reversibility
of the simple random walk imply that for any finite
$B \subset V$ we have 
\eqn{e:inf-hit}
{ \lim_{K \to \infty} \sup_{z \not\in B(o,K)}
    \P [ \xi_B < \infty \,|\, S(0) = z ]
  = 0. } 
Let $\{ S^{(1)}(n) \}_{n \ge 0}$ and $\{ S^{(2)}(n) \}_{n \ge 0}$ 
be independent simple random walks on $G$, with possibly
different initial states. It is easy to see that 
Assumption \ref{ass:Gcond}(ii) and transitivity implies 
\eqnst
{ \P \left[ S^{(1)}(0,\infty) \cap S^{(2)}(0,\infty) = \es 
  \,|\, S^{(1)}(0) = o = S^{(2)}(0)\right]
  =: \delta_0 
  > 0. }
We show that we also have
\eqn{e:assump}
{ \inf_{x \not= y} \P \left[ S^{(1)}[0,\infty) \cap S^{(2)}[0,\infty) = \es 
  \,|\, S^{(1)}(0) = x,\, S^{(2)}(0) = y \right]
  =: \delta 
  > 0. }
By transitivity, we may assume $x = o$. By \eqref{e:inf-hit}, we
can find $K_0$ such that $\dist(o,y) \ge K_0$ implies 
$\P [ \xi^{(2)}_o < \infty \,|\, S^{(2)}(0) = y ] < \delta_0/2$,
where $\xi^{(2)}_o$ denotes the hitting time of $o$ by $S^{(2)}$.
Let $\{ S^{(3)}(n) \}_{n \ge 0}$ and $\{ S(n) \}_{n \ge 0}$ 
be a third and a fourth independent simple
random walk, both starting at $o$. Assume the event 
\eqnst
{ A_{1,3}
  = \{ S^{(1)}(0,\infty) \cap S^{(3)}(0,\infty) = \es \}. }
By L\'evy's $0$--$1$ law, a.s.~on the event $A_{1,3}$ we have
\eqnst
{ \lim_{n \to \infty} 
     \P [ S^{(3)}[0,\infty) \cap S^{(1)}[0,\infty) = \es \,|\, 
     S^{(3)}[0,n],\, S^{(1)} ] 
  = 1. }
In particular, a.s.~on $A_{1,3}$ the random variables
\eqnsplst
{ X_{1,3}
  &= \inf_{z \in S^{(3)}(0,\infty)} 
    \P [ S[0,\infty) \cap S^{(1)}[0,\infty) = \es \,|\, 
    S(0) = z,\, S^{(3)} ] \\
  X_{3,1}
  &= \inf_{z \in S^{(1)}(0,\infty)} 
    \P [ S[0,\infty) \cap S^{(3)}[0,\infty) = \es \,|\, 
    S(0) = x,\, S^{(1)} ] }
are positive. Hence we can find $c > 0$ and 
$0 < \delta' < \delta_0/4$ such that 
\eqn{e:good-event}
{ \P [ A_{1,3},\, X_{1,3} \ge c,\, X_{3,1} \ge c,\, 
     \xi^{(2)}_o = \infty ]
  \ge \delta'. }
On the event in \eqref{e:good-event}, either 
$S^{(2)}$ never hits $S^{(1)}[0,\infty) \cup S^{(3)}[0,\infty)$, 
or if it hits one of the paths, then with conditional 
probability at least $c$, its continuation from the first 
hitting point does not hit the other path.
By symmetry of the roles of $S^{(1)}$ and $S^{(3)}$, we get
$\P [ S^{(1)}[0,\infty) \cap S^{(2)}[0,\infty) = \es ] \ge c \delta'$.
The statement involving all $y \not= o$ now follows.

We continue with the proof of statement (1).
The probability that $S[0,M]$ traces out exactly 
$\beta_1$ or $\beta_2$ is positive. Assume that this
occurs. Consider some large $K$, and let $A_K$ be the event that 
$\{ S(M+k) \}_{k \ge 0}$ and $\{ S(-k) \}_{k \ge 0}$
follow the paths $\rho$ and $\pi$, respectively, constructed in 
Lemma \ref{lem:cycle} until they both leave $B(o,K)$,
at vertices $x$ and $y$. By \eqref{e:inf-hit}, we can choose $K$ 
large enough so that 
\eqn{e:notcomeback}
{ \sup_{z \not\in B(o,K)} 
    \P [ \xi_{G_0} < \infty \,|\, S(0) = z ] 
  \le \delta/4. }
Due to \eqref{e:notcomeback} and \eqref{e:assump},
the conditional probability given $A_K$, that the walks 
satisfy the requirements (b) and (c) is at least 
$\delta/2$. This proves part (1).

\emph{Proof of (2).} We want to apply the ergodic 
theorem. A technical difficulty is that there
may be no canonical way to ``translate'' a vertex 
$x \in V$ back to $o$. Hence after shifting the path by 
$\phi_x^{-1}$, we average over $\AUT_o$,
which is possible, since $\AUT_o$ is compact.
This way we can define a certain path-valued 
stationary process. Let $\Psi_0, \Psi_1, \dots$ be
an i.i.d.~sequence, independent of the random walk, 
with each element distributed according to $\la$. Put 
\eqnst
{ X_k(m)
  := \Psi_k \phi_{S(Mk)}^{-1} S(Mk + m), 
     \quad -\infty < m < \infty,\, k \ge 0. }
We claim that the path-valued sequence 
$\{ X_k(\cdot) \}_{k \ge 0}$ 
is stationary and mixing (on the space of paths we consider
the topology of pointwise convergence and the induced Borel
$\sigma$-field). The somewhat tedious 
proof of this intuitive claim is deferred to the Appendix.

The proof of (2) is now straightforward from part (1) 
and the ergodic theorem, noting that $j$ is good if and 
only if $0$ is good relative to the path $X_j$.
\end{proof}

A time $-\infty < j < \infty$ is called \emph{loop-free}
for $S$ if $S(-\infty,j] \cap S(j,\infty) = \es$. 
The significance of loop-free points is that loop-erasure
on the two sides of a loop-free point do not influence
each other. Note that if $k \ge 0$ is good, then 
$kM$ and $(k+1)M$ are loop-free. This observation and 
Lemma \ref{lem:goodblocks}
immediately implies the following lemma.

\begin{lemma}
\label{lem:loopfree}
Suppose the graph $G = (V,E)$ satisfies 
Assumption \ref{ass:Gcond}(i)--(ii). There exists
$b' > 0$ such that 
\eqnst
{ \P \left[ \parbox{7.5cm}{$\exists K_0$ 
      $\forall K \ge K_0$
      there are at least $b' K$
      loop-free points among $0, \dots, K-1$} \right] 
  = 1. }  
\end{lemma}

The lower bound on the fluctuations can now be achieved
by conditioning on ``all information outside the 
good blocks''. In order to make this precise, for each
good index $k \ge 0$, we choose $\psi_k \in \AUT_o$ 
such that $\psi_k \phi^{-1}_{S(kM)} B_k \in \{ \beta_1, \beta_2 \}$.
Note that since $\beta_1$ and $\beta_2$ both traverse 
$G_0$, if $\psi'_k$ is another such automorphism, then
$\psi'_k \psi_k^{-1}$ fixes $G_0$ pointwise. In particular,
$\psi_k^{-1}\vert_{G_0} \equiv \psi'^{-1}_k\vert_{G_0}$, where
$\vert_{G_0}$ denotes restriction to $G_0$.
We define the $\sigma$-algebra 
$\cG$ generated by the following random objects: 
\eqnspl{e:conditioning}
{ &S(kM), \quad k \ge 0; \\
  &Y_k 
  := I[\text{$k$ is good}], \quad k \ge 0; \\
  &\psi_k^{-1}\vert_{G_0} \text{ for $k \ge 0$ such that $Y_k = 1$}; \\
  &\text{the \emph{paths} $B_{k'}$} 
  = [ S(k'M), S(k'M+1), \dots, S((k'+1)M) ], \\ 
  &\qquad \text{for $k' \ge 0$ such that $Y_{k'} = 0$}; \\  
  &S(-\infty,0]. }

\begin{lemma}
\label{lem:condindep}
Given $\cG$, the good blocks are conditionally 
independent, and conditional on $\cG$, such a block $B_k$ takes 
the values $\Phi_k \beta_1$ and $\Phi_k \beta_2$ with probabilities
$1/2$ each, 
for some $\cG$-measurable automorphisms $\{ \Phi_k \}$ that take 
$o$ to $S(kM)$, respectively.
\end{lemma}

\begin{proof}
We know that almost surely there are infinitely many good 
indices. Fix $N \ge 2$.
Consider the class $\cP_N$ of events of the form:
\eqnsplst
{ E
  &= \{ S(-j) = y_j,\, j = 1, \dots, J; \\
  &\qquad S(kM) = z_k,\, k = 0, \dots, K; \\
  &\qquad \text{$k \in I$ are good}; \\
  &\qquad \text{$k' \in \{0, \dots, K-1\} \setminus I$ are not good} \\
  &\qquad \text{$\psi_k \phi^{-1}_{S(kM)} B_k 
     \in \{ \beta_1, \beta_2 \}$ for $k \in I$}; \\
  &\qquad \text{$\psi_k^{-1}\vert_{G_0} = \alpha$}; \\
  &\qquad \text{$B_{k'} = b_{k'}$ for 
    $k' \in \{0, \dots, K-1\} \setminus I$} \}, }
where $J$, $K \ge N$, $z_k$, $I \subset \{0, \dots, K-1\}$, $|I| = N$, 
$\alpha$, and $b_{k'}$ are fixed. Let $A_k^{\eps_k}$ be the event
$\{ \psi_k \phi_{z_k}^{-1} B_k = \beta_{\eps_k} \}$, where
$\eps_k \in \{1,2\}$ for $k = 0, \dots, K-1$. 
By decomposing the path of $S$ into 
$\{ S(-n) \}_{n \ge 0}$, the blocks $B_k$, $k = 0, \dots, K-1$, and
$\{ S(n) \}_{n \ge KM}$, we see that 
\eqnst
{ \P \left[ \left( \cap_{k \in I} A^{\eps_k}_k \right) \cap E \right]
  = \left( \prod_{k \in I} \frac{1}{2} \right) \P [ E ]. }
Since $\cP_N$ is closed under intersection, and generates 
$\cG$, this implies conditional independence of the first 
$N$ good blocks. The Lemma follows.
\end{proof}

\begin{proof}[Proof of Proposition \ref{prop:fluct}]
Let $\eps > 0$ be given.
Due to Lemma \ref{lem:highDcoupling}, there exists
a finite $B \subset V$ such that for all 
large enough $n$, with probability at least $1 - \eps$, 
we have $v^{(i)}_{n,x} \in B$ for $i = 1, \dots, K_{n,x}$.
Hence the Proposition will follow, once we show that 
for any $v, w \in B$ we have
\eqn{e:not-close}
{ \lim_{M \to \infty} \limsup_{n \to \infty}
    \P \left[ \pi_{n,v} \cap \pi_{n,w} = \es,\, 
    \left | d_{t_n}(v,s) - d_{t_n}(w,s) \right| 
    \le M \right]
  = 0. }
Let $S^{1}$ and $S^{2}$ be independent simple random 
walks starting at $v$ and $w$, respectively.
Let $\gamma^{i} := \LE(S^{i}[0,\tau^{i}_n])$.
We apply Lemma \ref{lem:condindep} to the random walk 
$S = S^{1}$. Let 
\eqnst
{ T_n
  := \sup \{ j \ge 0 : \text{$B_j \subset V_n$ and $j$ is good} \} }
be the index of the last good block completed before $\tau^{1}_{n}$.
Let the set of good indices be: $\{ g_1, g_2, \dots \}$
We define the stretches between good blocks: we let
\eqnst
{ \rho_0 
  := [S^{1}(0), S^{1}(1), \dots, S^{1}(g_1 M)], } 
and for $k \ge 1$ we let 
\eqnst
{ \rho_k 
  := [S^{1}((g_k+1)M), S^{1}((g_k+1)M +1), \dots, S^{1}(g_{k+1} M)]. }
Observe that loop-erasure of the $\rho_k$'s do not interfere 
with each other, due to item (c) of Definition \ref{def:good}. 
Hence $\gamma^{1} = \LE ( S^{1}[0, \tau^1_{n}] )$ is the concatenation of:
\eqnsplst
{ &\LE(\rho_0), \LE(B_{g_1}), \LE(\rho_1), \LE(B_{g_2}), \dots, \\
  &\LE(\rho_{T_n-1}), \LE(B_{T_n}), \LE(S^{1}[(T_n+1)M,\tau^{1}_{n})). }
Due to Lemma \ref{lem:goodblocks} (ii), for any $K$ we have 
\eqnst
{ \lim_{n \to \infty} \P [ T_n \ge K ]
  = 1. }
Now condition on the random walk $S^{2}$, condition 
on the set of good indices and the bad blocks of $S^{1}$ 
up to exit from $V_n$, and condition on the event 
$\{ T_n \ge K \}$. Then
\eqnst
{ \length(\gamma^{1}) - \length(\gamma^{2})
  = (Y_0 - \length(\gamma^{2})) 
    + \sum_{j=1}^{K} Y_j, }
where the value of the first term is fixed by the conditioning, 
and the $Y_j$ are conditionally i.i.d.~with positive variance.
By the local central limit theorem \cite{Spitzer}, we get
\eqnst
{ \P \left[ \left| \length(\gamma^{1}) - 
     \length(\gamma^{2}) \right| \le M \right]
  \le \frac{c M}{\sqrt{K}}. }
Letting $K \to \infty$ implies the claim in \eqref{e:not-close}, 
and hence the Proposition follows.
\end{proof}

\subsection{Asymptotic uniformity of the permutation}
\label{ssec:unif}

In this section we prove Proposition \ref{prop:unif},
and complete the proof of Theorem \ref{thm:trans}.

Let $k \ge 1$ be fixed and let $(F^{(i)},v^{(i)})$,
$1 \le i \le k$ be a fixed sequence of finite
rooted trees in $G$. We will use Wilson's method
to generate samples $\bt$ and $t_n$ from 
the measures $\mu$ and $\mu_n$. The set-up is the same
as in Section \ref{ssec:fluct}, that is, we use the
same random walks $S^j$ started at the vertices
\eqnst
{ u_1 = v^{(1)}, \dots, u_k = v^{(k)}, u_{k+1}, \dots, u_L, }
where $u_{k+1}, \dots, u_L$ is an enumeration of
$\cup_{x \in A} \cN_x$. Recall that $T^i$ and $T^i_n$ denote 
the hitting times of $\cF_{i-1}$ and $\cF_{n,i-1}$,
respectively, by $S^i$.

Let $B_1$ denote the set of vertices in $\cup_{i=1}^k F^{(i)}$.
Let 
\eqnst
{ C 
  := \left( \cap_{i=1}^k \{ T^i = \infty \} \right)
     \cap \{ \cF_k \cap B_1 = \{ v^{(1)}, \dots, v^{(k)} \} \}. }
Let 
\eqnst
{ C'
  := \left\{ \cup_{j=k+1}^L \LE(S^j[0,T^j]) 
     = \cup_{i=1}^k F^{(i)} \right\}. }
Observe that as long as $(F^{(i)},v^{(i)})_{i=1}^k$ is a possible
sequence for $(F^{(i)}_A,v^{(i)}_A)_{i=1}^k$, we have
\eqn{e:CC'}
{ C \cap C' 
  = \left\{ K_A = k,\, 
    (F^{(i)}_A, v^{(i)}_A) = (F^{(i)}, v^{(i)}),\, 
    1 \le i \le k \right\}. }
We also introduce
\eqnsplst
{ C_n
  &:= \left( \cap_{i=1}^k \{ T^i_n = \tau^i_n \} \right)
     \cap \{ \cF_{n,k} \cap B_1 = \{ v^{(1)}, \dots, v^{(k)} \} \} \\ 
  C'_n
  &:= \left\{ \cup_{j=k+1}^L \LE(S^j[0,T^j_n]) 
     = \cup_{i=1}^k F^{(i)} \right\}, }
and observe that 
\eqn{e:CnC'n}
{ C_n \cap C'_n 
  = \left\{ K_{n,A} = k,\, 
    (F^{(i)}_{n,A}, v^{(i)}_{n,A}) = (F^{(i)}, v^{(i)}),\, 
    1 \le i \le k \right\}. }

Here is the outline of the proof. The restriction involving 
$B_1$ has little effect on the walks $S^i$, $i = 1, \dots, k$,
once they are far away from $B_1$, and likewise, the 
condition $\{ T^i = \infty \}$, $i = 1, \dots, k$. Therefore,
for some large $n'$, once these walks leave $V_{n'}$, they 
can be treated as independent. The point where Assumption 
\ref{ass:Gcond}(iv) (bounded harmonic functions are constant)
becomes crucial, is to show that the walks can also be treated 
as having the same distribution. Namely, we show that 
Assumption \ref{ass:Gcond}(iv) implies that for some $n'' > n'$,
the exit measures of the walks on $\partial V_{n''}$ are nearly 
identical in total variation distance. Therefore, their 
continuations are nearly i.i.d. This will imply the near 
uniformity of $\sigma_{n,A}$ for $n \gg n''$.

Let $\eps > 0$ be given.
As in the proof of Lemma \ref{lem:highDcoupling}, we deduce that 
\eqn{e:not-intersect}
{ \lim_{n \to \infty} 
     \P \left[ \left( \cap_{i=1}^k \{ T^i = \infty \} \right)
     \bigtriangleup \left( \cap_{i=1}^k \{ T^i_n = \tau^i_n \} \right) 
     \right]
  = 0, }
where $\bigtriangleup$ denotes symmetric difference.
Letting $\hat{\tau}^i_{B_1}$ denote the last visit by $S^i$ 
to the set $B_1$, transience implies that for each $i = 1, \dots, k$
we have
\eqn{e:not-come-back}
{ \lim_{n \to \infty} \P \left[ S^i[\tau^i_n,\infty) \cap 
      S^i[0,\hat{\tau}^i_{B_1}] \not= \es \right]
  = 0. }
It follows from \eqref{e:not-intersect} and \eqref{e:not-come-back}
that there exists $n_1$ such that for all $n \ge n_1$ we have
\eqn{e:small-diff1}
{ \P \left[ C \bigtriangleup C_n \right]
  < \eps. }
Since on the event $C \cap C'$ we have $T^j < \infty$ 
for $j = k+1, \dots, L$, we can find a large enough finite 
set $B_2 \subset V$ ($B_2 \supset B_1$) such that 
with $G_1 := \cap_{j=k+1}^L \{ S^j[0,T^j] \subset B_2 \}$
we have
\eqn{e:not-go-out}
{ \P \left[ C \cap C' \cap G_1^c \right]
  < \eps. }
Let 
\eqnst
{ G_{n,2}
  := \cap_{i=1}^k \left\{ S^i[\tau^i_n,\infty) \cap 
     S^i[0,\hat{\tau}^i_{B_2}] = \es \right\}. }
Applying \eqref{e:not-come-back} for $B_2$ in place of $B_1$,
we get that there exists $n_2$ such that for all $n \ge n_2$
we have
\eqn{e:not-come-back2}
{ \P \left[ G_{n,2} \right] 
  > 1 - \eps. }
It follows from \eqref{e:small-diff1}, \eqref{e:not-go-out} and
\eqref{e:not-come-back2} 
that there exists $n_3$ such that for all $n \ge n_3$ we have
\eqn{e:small-diff2}
{ \P [ (C \cap C') \bigtriangleup (C_n \cap C'_n) ] 
  < 3 \eps. }
Let $S$ denote a simple random walk independent of the $S^j$'s.
L\'evy's $0$--$1$ law implies that for each $i = 2, \dots, k$,
almost surely we have
\eqnst
{ \lim_{n \to \infty} \P \Bigg[ S[0,\infty) \cap 
     \Bigg( \bigcup_{\substack{1 \le j \le k\\ j \not= i}} 
     S^j[0,\infty) \Bigg) = \es \,\Bigg|\,
     S(0) = S^i(\tau^i_n),\, S^j,\, j = 1, \dots, k,\, j \not= i \Bigg] 
  = 1. }
Hence we can find $n_4$ such that with
\eqnst
{ G_{n,3}
  := \bigcap_{i=1}^k \Bigg\{ S^i[\tau^i_n,\infty) 
     \cap \Bigg( \bigcup_{\substack{1 \le j \le k\\ j \not= i}} 
     S^j[0,\infty) \Bigg) = \es \Bigg\} }
for all $n \ge n_4$ we have
\eqn{e:not-intersect2}
{ \P \left[ G_{n,3} \right] 
  > 1 - \eps. }

Let $n' := \max \{ n_1, n_2, n_3, n_4 \}$.
Given $D \subset V$, $B \subset \partial D$ and 
$z \in \bar{D} := D \cup \partial D$, let
\eqnst
{ h_D(z,B)
  := \P [ S(\tau_D) \in B \,|\, S(0) = z ] }
denote the exit measure of simple random walk on the boundary of $D$.
Note that $h_D(\cdot,B)$ is harmonic in $D$ for any 
$B \subset \partial D$, and $0 \le h_D(z,B) \le 1$.
Here we write $\partial D = \{ y \in V \setminus D : 
\text{$y \sim x$ for some $x \in D$} \}$.
We show that we can find an index $n'' > n'$ such that 
\eqn{e:couple}
{ \sup_{z_1, z_2 \in \partial V_{n'}} 
     \| h_{V_{n''}}(z_1, \cdot) - h_{V_{n''}}(z_2, \cdot) \|
  \le \eps, }
where $\| \cdot \|$ denotes total variation distance.
Indeed, if this was not the case we could find 
$z_1, z_2 \in \partial V_{n'}$, and a sequence
$r_1 < r_2 < \dots$ and subsets $A_i \subset \partial V_{r_i}$
such that 
\eqn{e:delta}
{ \left| h_{V(r_i)}(z_1,A_i) - h_{V(r_i)}(z_2,A_i) \right| 
  \ge \eps, \quad i = 1, 2, \dots. }
By passing to a subsequence, we may assume that the limit
\eqnst
{ h(z) 
  := \lim_{i \to \infty} h_{V(r_i)}(z,A_i) } 
exists for all $z \in V$. From \eqref{e:delta} we have
$|h(z_1) - h(z_2)| \ge \eps$. However, $h$ is a bounded harmonic 
function, so it must be constant by Assumption \ref{ass:Gcond}(iv).
This contradiction proves \eqref{e:couple}.

Consider now $n > n''$, and let 
\eqnst
{ f(y)
  := \P [ S(\tau_{n''}) = y \,|\, S(0) = o ], \quad 
     y \in \partial V_{n''}. }
It follows from \eqref{e:couple} that 
\eqn{e:close}
{ \| f(\cdot) - h_{V_{n''}}(x_i,\cdot) \| 
  \le \eps, \quad i = 1, \dots, k. }
Hence, for any $x_1, \dots, x_k \in \partial V_{n'}$ there exists
a coupling $g_{x_1,\dots,x_k}(z_1,\dots,z_k,y_1,\dots,y_k)$ with 
marginals 
\eqnsplst
{ \sum_{y_1,\dots,y_k} g_{x_1,\dots,x_k}(z_1,\dots,z_k,y_1,\dots,y_k)
  &= \prod_{i=1}^k h_{V_{n''}}(x_i, z_i) \\
  \sum_{z_1,\dots,z_k} g_{x_1,\dots,x_k}(z_1,\dots,z_k,y_1,\dots,y_k)
  &= \prod_{i=1}^k f(y_i), }
where 
\eqnst
{ \sum_{z_1 = y_1, \dots, z_k = y_k} 
      g_{x_1,\dots,x_k} (z_1,\dots,z_k,y_1,\dots,y_k)
  \ge 1 - O(\eps). }
Let $\{ \tilde{S}^i(n) \}_{n \ge 0}$, $i = 1, \dots, k$
be independent simple random walks with 
initial distribution $f$. We couple the initial 
distribution of the $\tilde{S}^i$'s to the distribution 
of $S^i(\tau^i_{n''})$'s using $g$, where $x_i = S^i(\tau^i_{n'})$.
In this coupling, we have $S^i(\tau^i_{n''}+m) = \tilde{S}^i(m)$,
$m \ge 0$ with probability at least $1 - O(\eps)$. 

We define the random permutation
$\tilde{\sigma} \in \Sigma_k$ by the condition
\eqn{e:def-perm}
{ \length( \LE(\tilde{S}^{\tilde{\sigma}(1)}
    [0,\tilde{\tau}^{\tilde{\sigma}(1)}_n]) )
  < \dots
  < \length( \LE(\tilde{S}^{\tilde{\sigma}(k)}
    [0,\tilde{\tau}^{\tilde{\sigma}(k)}_n]) ). }
Here, if there are ties, we break them in a uniformly
random way. That is, if 
$\{ i_1, \dots, i_r \} \subset \{ 1, \dots, k \}$
is a maximal set of indices such that the loop-erasures
of the paths $\tilde{S}^{i_j}[0,\tilde{\tau}^{i_j}_n]$,
$j = 1, \dots, r$ have equal lengths, we pick an ordering 
on them uniformly at random, and use that ordering 
in \eqref{e:def-perm}. This way of breaking ties ensures 
that $\tilde{\sigma}$ is exactly uniformly distributed on
$\Sigma_k$.

It follows from Lemma \ref{lem:loopfree}, and the almost 
sure finiteness of $\tau^i_{n''}$, that there 
exist $n_5 > n''$ and an $M_1 < \infty$ such that with 
\eqnst
{ G_{n,4}
  := \cap_{i=1}^k 
     \{ \text{there exists a loop-free point for $S^i$
        in $[\tau^i_{n''},\tau^i_{n''}+M_1]$} \} }
for all $n \ge n_5$ we have
\eqn{e:loopfree-exists}
{ \P [ G_{n,4} ]
  > 1 - \eps. }
Occurrence of the event $G_{n,4}$ ensures that when $n \gg n''$, 
most of the length of $\LE(S^i[0,\tau^i_n])$ comes from
the length of $\LE(\tilde{S}^i[0,\tilde{\tau}^i_n])$, for
$i = 1, \dots, k$. In particular, there exists a 
deterministic $M_2 = M_2(M_1,n'')$, such that 
whenever $C_n \cap G_{n,4} \cap \{ |d^{(i,j)}_{n,A}| > M_2 \}$
occurs, we have $\sigma_{n,A} = \tilde{\sigma}$.

We are ready to start analyzing the event on the left hand
side of \eqref{e:uniperm}.
A straightforward computation shows that for any events
$A \in \sigma( S^i[0,\tau^i_{n'}], i = 1, \dots, k )$
and $B \in \sigma( \tilde{S}^i[0,\infty), i = 1, \dots, k )$
we have 
\eqnspl{e:near-indep}
{ &| \P [ A \cap B ] - \P [ A ] \P [ B ] | \\
  &\qquad \le \sup_{x_1,\dots,x_k} \sum_{\substack{z_1,\dots,z_k\\y_1,\dots,y_k}}
      \left| g_{x_1,\dots,x_k} (z_1,\dots,z_k,y_1,\dots,y_k) 
      - \prod_{i=1}^k \delta(z_i,y_i) f(y_i) \right| \\
  &\qquad \le O(\eps). } 
We apply this with $A = C_{n'}$ and $B = \{ \tilde{\sigma} = s \}$,
where $s \in \Sigma_k$ is fixed.
To be precise, due to the breaking of ties for $\tilde{\sigma}$, 
this $B$ is defined on a slightly larger $\sigma$-field than 
in \eqref{e:near-indep}. But this has no consequence. 
Using \eqref{e:near-indep} and 
\eqref{e:small-diff1}, for $n > n''$ we obtain
\eqn{e:approx-indep}
{ \P [ C_{n'} \cap \{ \tilde{\sigma} = s \} ]
  = \frac{1}{k!} \P [ C_{n'} ] + O(\eps)
  = \frac{1}{k!} \P [ C ] + O(\eps). }
Our goal now is to show that on a slightly different event
$\tilde{\sigma}$ can be replaced by $\sigma_{n,A}$.

Due to Proposition \ref{prop:fluct}, we can find 
$n_6 > n''$ such that for all $n \ge n_6$ we have
\eqn{e:largediff}
{ \P \left[ \min_{1 \le i < j \le k} d^{(i,j)}_{n,A} > M_2 \right]
  \ge 1 - \eps. }
Consider for $n \ge \max \{ n_5, n_6 \}$ the event 
\eqn{e:all-occur}
{ \tilde{C}_n
  := C_{n'} \cap G_{n',2} \cap G_{n',3} \cap G_{n,4} \cap 
     \left\{ \min_{1 \le i < j \le k} d^{(i,j)}_{n,A} > M_2 \right\}. }
Observe that $\tilde{C}_n \subset C_n$ and that 
on $\tilde{C}_n$, we have $\sigma_{n,A} = \tilde{\sigma}$.
Due to the estimates \eqref{e:small-diff1},
\eqref{e:not-come-back2}, \eqref{e:not-intersect2}
and \eqref{e:largediff}, for $n \ge \max \{ n_5, n_6 \}$ we have
\eqn{e:modified}
{ \P [ \tilde{C}_n \cap \{ \sigma_{n,A} = s \} ]
  = \P [ C_{n'} \cap \{ \tilde{\sigma} = s \} ] + O(\eps). }
The presence of the event $G_{n,2}$ in \eqref{e:all-occur} ensures 
that on the event $\tilde{C}_n \cap \{ \sigma_{n,A} = s \}$, we have 
$\cF_k \cap B_2 = \cF_{n,k} \cap B_2$.  Therefore,
\eqn{e:conditional}
{ \P \big[ C'_n \cap G_1 \,\big|\, \tilde{C}_n 
     \cap \{ \sigma_{n,A} = s \} \big]
  = \P [ C' \cap G_1 \,|\, C ]
  = \P [ C' \,|\,C ] + O(\eps). }
It follows from \eqref{e:approx-indep}, \eqref{e:modified} and
\eqref{e:conditional} that 
\eqn{e:almost}
{ \P \left[ \tilde{C}_n \cap \{ \sigma_{n,A} = s \} 
    \cap C'_n \cap G_1 \right]
  = \frac{1}{k!} \P [ C \cap C' ] + O(\eps). }
Since $\P [ \tilde{C}_n \bigtriangleup C_n ] = O(\eps)$ and
$\P [ G_1^c \cap C \cap C' ] < \eps$, \eqref{e:almost}
implies that 
\eqnst
{ \P \left[ C_n \cap C'_n \cap \{ \sigma_{n,A} = s \} \right]
  = \frac{1}{k!} \P [ C \cap C' ] + O(\eps). }
Comparing with \eqref{e:CC'} and \eqref{e:CnC'n},
this completes the proof of the Proposition.

\begin{proof}[Proof of Theorem \ref{thm:trans}]
Due to Lemma \ref{lem:highDcoupling}, for any $\eps > 0$ 
there exists a finite $B \subset V$ such that 
\eqnst
{ \liminf_{n \to \infty} \P \left[ 
     \cup_{i=1}^{K_{n,A}} (F^{(i)}_x,v^{(i)}_x) 
     \subset B \right]
  \ge 1 - \eps. }
Hence we can restrict our attention to the finite collection 
of rooted trees $(F,v)$ that lie inside $B$. Let 
$(F^{(1)}, v^{(1)}), \dots, (F^{(k)},v^{(k)})$ be a possible
value of $(F^{(1)}_A,v^{(1)}_A), \dots (F^{(k)}_A,v^{(k)}_A)$,
with $K_A = k$, such that all these trees lie inside $B$.

Lemma \ref{lem:highDcoupling} shows that in a suitable
coupling, the events 
$(F^{(i)}_{n,A},v^{(i)}_{n,A})_{i=1}^k = (F^{(i)},v^{(i)})_{i=1}^k$
and $(F^{(i)}_{A},v^{(i)}_{A})_{i=1}^k = (F^{(i)},v^{(i)})_{i=1}^k$
are asymptotically equal, when this occurs, we have
$(F^{(i)}_{n,x},v^{(i)}_{n,x})_{i=1}^{K_{n,x}} = (F^{(i)}_x,v^{(i)}_x)_{i=1}^{K_x}$
for all $x \in A$.
Proposition \ref{prop:fluct} implies that for all 
$x \in A$ and for large enough $n$ the condition 
\eqref{e:diam-large} of Lemma \ref{lem:only-perm}
holds with high probability. This implies that 
with high probability, the permutations 
$\{ \sigma_{n,x} \}_{x \in A}$ are determined by 
$\sigma_{n,A}$. Moreover, the dependence
of the collection $\{ \sigma_{n,x} \}_{x \in A}$ on 
$\sigma_{n,A}$ is given by the same (deterministic) 
function as the dependence of $\{ \sigma_x \}_{x \in A}$ 
on $\sigma_A$.
Proposition \ref{prop:unif} implies that conditioned on 
$(F^{(i)}_A,v^{(i)}_A) = (F^{(i)},v^{(i)})$, $i = 1, \dots, k$,
the distribution of $\sigma_{n,A}$ is close to uniform.
This implies that for each $x \in A$, the joint distribution of 
$\{ \sigma_{n,x} \}_{x \in A}$ is close to the joint distribution 
of $\{ \sigma_x \}_{x \in A}$.

The above considerations, Lemma \ref{lem:only-perm}, and the 
definition of $\eta$ in \eqref{e:etax} imply that the joint 
distribution of $\{ \eta_{n,x} \}_{x \in A}$ converges to the 
joint distribution of $\{ \eta_x \}_{x \in A}$ as $n \to \infty$.
Hence the Theorem follows.
\end{proof}

\section{Infinite volume limits on regular trees}
\label{sec:limitsIII}

In this section we consider infinite $d$-regular trees.
The paper \cite{MRS02} proved the existence of the
limit $\nu_n \Rightarrow \nu$ along sufficiently
regular exhaustions (see condition (24) there).
It was also claimed that the limit exists along any
exhaustion, however this does not follow from the
arguments in \cite{MRS02} (note that statement (25) 
of \cite{MRS02} does not imply the Cauchy net property 
claimed there). In this section we prove the general
convergence result. Note that the proof of 
Theorem \ref{thm:trans} does not apply to the infinite 
$d$-regular tree, for more than one reason: 
Assumption \ref{ass:Gcond}(iv) is not satisfied, 
and there is no fluctuation in the lengths of paths,
so Proposition \ref{prop:fluct} fails.
Nevertheless, the Majumdar-Dhar bijection can still be used 
to show that $\nu_n \Rightarrow \nu$ along any exhaustion. 

Let $G = (T^d,E)$ be the infinite $d$-regular tree, with
$d \ge 3$. We will denote by $o$ an arbitrary fixed vertex 
of $G$.

\begin{theorem}
\label{thm:tree}
For any $d \ge 3$ and any exhasution 
$V_1 \subset V_2 \subset \dots \subset T^d$, we have
$\nu_n \Rightarrow \nu$ for a unique measure $\nu$,
independent of the exhaustion.
\end{theorem}

We begin with some preparations for the proof.
Fix a finite $A \subset T^d$. We need to consider the convergence of 
$\nu_{G_n} [ \eta_{n,x} = h_x,\, x \in A ]$, as $n \to \infty$, 
where $h \in \{0,1,\dots,d-1\}^A$ is fixed.

Let $\{ S(n) \}_{n \ge 0}$ denote a simple random walk 
in $T^d$. Let $\tau_n := \inf \{ k \ge 0 : S(k) \not\in V_n \}$,
and for $B \subset T^d$, let 
$\xi_B := \inf \{ k \ge 0 : S(k) \in B \}$.
The following notation will be useful: given 
$x \in \partial A$ and $V_n \supset A$, let 
\eqnsplst
{ q_{n,x} 
  := \P [ \tau_n < \xi_A \,|\, S(0) = x ]. }
Given a self-avoiding path $\sigma$ from $x$ to $V_n^c$
that does not visit $A$, we also define:
\eqnspl{e:path-event}
{ q_{n,x}(\sigma) 
  := \P [ \text{$\tau_n < \xi_A$ and $\LE(S[0,\tau_n]) = \sigma$} \,|\, 
     S(0) = x ]. }

Recall that $\cT_{G_n}$ is the set of spanning trees of $G_n$.
We will orient edges towards the sink, and view trees as
arrow configurations. For $t_n \in \cT_{G_n}$ let
\eqnst
{ \cC(t_n) 
  := \{ y \in \partial A : \text{$\exists$ $e \in t_n$ 
     such that $e_- = y$ and $e_+ \in A$} \}. }
Note that we always have $\cC(t_n) \subsetneqq \partial A$.
We classify trees according to the value of $\cC$. 
Fix $C \subsetneqq \partial A$, and consider trees $t_n$ with 
$\cC(t_n) = C$. In any such tree, the path from a vertex 
$y \in (\partial A) \setminus C$ to $s$, that is the path 
$\pi_{n,y}(t_n)$, does not visit $A$. 
Due to Lemma \ref{lem:depend}, the occurrence or not of the event 
$\{ \eta_{n,x} = h_x,\, x \in A \}$ depends on: the lengths 
of $\{ \pi_{n,y}(t_n) \}_{y \in (\partial A) \setminus C}$ 
and the position of arrows with tails in $A$.
We will refer to the latter simply as ``the arrows in $A$''.
We denote by $m_{n,y}$ the length of $\pi_{n,y}$.
The key to convergence is a remarkable symmetry property 
of the bijection stated in the next two lemmas.

\begin{lemma}
\label{lem:sym1}
For any $C \subset \partial A$ and $\{ h_x \}_{x \in A}$, 
the following alternative holds. Either \\
(A) for any collection $m_{n,y} \ge 1$, $y \in (\partial A) \setminus C$, 
the event $\{ \eta_{n,x} = h_x,\, x \in A \}$ does not occur for 
any choice of arrows in $A$; or \\
(B) for any collection $m_{n,y} \ge 1$, $y \in (\partial A) \setminus C$,
the event $\{ \eta_{n,x} = h_x,\, x \in A \}$ occurs for exactly one 
choice of arrows in $A$.
\end{lemma}

\begin{lemma}
\label{lem:sym2}
Suppose that $C \subset \partial A$ and that Case (B) holds 
in Lemma \ref{lem:sym1}. Let $\sigma_y$, 
$y \in (\partial A) \setminus C$ be fixed self-avoiding paths 
from each $y$ to $s$ that avoid $A$. Then
\eqnspl{e:event-with-paths}
{ &\mu_n \left[ \cC(t_n) = C;\, 
     \pi_{n,y}(t_n) = \sigma_y,\, y \in (\partial A) \setminus C;\, 
     \eta_{n,x}(t_n) = h_x,\, x \in A \right] \\
  &\qquad = f_{A,C} (q_{n,y'},\, y' \in \partial A) 
    \prod_{y \in (\partial A) \setminus C} q_{n,y}(\sigma_y) }
for some function $f_{A,C} : [0,1]^{\partial A} \to [0,1]$, whose
form only depends on the pair $(A,C)$, and not on
the $\sigma_y$'s. The statement extends to Case (A),
by taking $f_{A,C}$ to be the $0$ function.
\end{lemma}

\begin{proof}[Proof of Lemma \ref{lem:sym1}]
Consider the following auxiliary graph. We start with
the subgraph of $G$ induced by $A \cup \partial A$. For each 
$y \in (\partial A) \setminus C$ we glue a path of length $m_{n,y}$ 
at $y$. All glued on paths end at the common endpoint $s$, that 
serves as the sink. No new edges are added for vertices in $C$. 
We denote this graph by $G_{A,C}$ (the dependence on the $m_{n,y}$'s 
is suppressed in the notation).
Consider the following sandpile configuration $\eta(h)$ on $G_{A,C}$.
On the set $A$, $\eta(h)$ equals $h$, on $C$ it equals $0$, and on 
the rest of the vertices it equals $1$. It is easy to see using 
the Burning Algorithm, that whether $\eta(h) \in \cR_{G_{A,C}}$ or not
is independent of the values of $m_{n,y}$. We claim that if 
$\eta(h) \not\in \cR_{G_{A,C}}$ then the statements in Case (A) hold, 
and if $\eta(h) \in \cR_{G_{A,C}}$ then the statements in Case (B) hold.

Consider any $\eta_n \in \cR_{G_n}$, for which $\eta_{n,x} = h_x$, 
$x \in A$, and for which the Burning Algorithm produces a tree 
$t_n$ with $\cC(t_n) = C$, and paths $\pi_{n,y}$ with lengths $m_{n,y}$. 
We consider the burning of $\eta_n$ in $G_n$ in parallel to the burning 
of $\eta(h)$ in $G_{A,C}$. We show that inside 
$\partial A \cup A$, each site will burn at the same time in the 
two processes. 

Since the time of burning equals graph distance 
from the sink in the tree produced by the algorithm, in both 
configurations the first time when a vertex of $\partial A$ burns is 
$m_1 := \min \{ m_{n,y} : y \in (\partial A) \setminus C \}$.
Let $y_{1,1}, \dots, y_{1,r_1}$ be the list of $y$'s for which the
minimum is achieved. After time $m_1$, the status of 
vetices in the subtree of $V_n$ emanating from each $y_{1,i}$
away from $A$ has no influence on the burning of vertices in 
$A \cup \partial A$ (they have been disconnected by the burning
of the vertex $y_{1,i}$). Hence we may discard these subtrees from 
$V_n$ for the rest of the process.
Let $m_2 := \min \{ m_{n,y} : m_{n,y} > m_1,\, 
(y \in \partial A) \setminus C \}$, and let $y_{2,1}, \dots, y_{2,r_2}$
be the list of $y$'s for which the minimum is achieved.

We claim that at all times $m_1 \le m \le m_2$, the two burning 
processes agree in $A \cup \partial A$. We show this by 
induction on $m$. The claim holds for $m = m_1$, as in both 
processes precisely $y_{1,1}, \dots, y_{1,r_1}$ are burnt at time 
$m_1$. Assume the claim holds for some $m$ with $m_1 \le m < m_2$. 
Let $z \in A$ be a vertex that is unburnt at time $m$ (in 
both configurations, necessarily).
The equality $\eta(h)_z = h_z = \eta_{n,z}$ and the
induction hypothesis ensures that $z$ burns at time $m+1$ in 
$\eta(h)$ if and only if it burns in $\eta_n$.
Let now $z \in C$, and let $z' \in A$ be the unique neighbour
of $z$ in $A$. Since $[z,z'] \in t_n$, $z$ will burn in $\eta_n$ 
at time $m+1$ if and only if $z'$ burnt at time $m$. By the 
induction hypothesis, the latter occurs if and only if 
$z'$ burnt in $\eta(h)$ at time $m$. Then by the definition 
$\eta(h)_z = 0$ we get that this happens if and only if 
$z$ burns in $\eta(h)$ at time $m+1$. 
Finally, consider a vertex $z \in (\partial A) \setminus C$
that is unburnt at time $m$ (in both configurations, necessarily).
Let $z' \in A$ be its unique neighbour in $A$. 
Since $[z',z] \in t_n$, $z'$ burns after $z$ in $\eta_n$, and
hence by the induction hypothesis $z'$ is unburnt at time $m$
in both $\eta_n$ and $\eta(h)$. 
In $\eta_n$, $z$ will burn at time $m+1$ if and only if 
$m+1 = m_2 = m_{n,z}$, and $z = y_{2,i}$ for some $1 \le i \le r_2$.
Due to the definition $\eta(h)_z = 1$ and the fact that 
$z'$ is unburnt in $\eta(h)$ at time $m$, this is equivalent 
to $z$ burning in $\eta(h)$ at time $m+1$. This completes the 
induction. 

We can now iterate the above argument until there are 
no more burnable vertices in $A \cup \partial A$, showing
that the two burning processes are identical in 
$A \cup \partial A$.

The equality of the burning processes gives that if 
$\eta(h) \not\in \cR_{G_{A,C}}$, then there 
can be no tree with the given $h$, $C$ and $m_{n,y}$'s. 
If $\eta(h) \in \cR_{G_{A,C}}$, then there is exactly
one possible arrow configuration in $A$, namely the one 
given by the burning of $\eta(h)$ in $G_{A,C}$ (here 
we use the same $\alpha_{P,K}$'s in the graphs 
$G_{A,C}$ and $G_n$). This completes the proof.
\end{proof}

\begin{proof}[Proof of Lemma \ref{lem:sym2}]
Consider the following auxiliary weighted graph $G' = G'_{A,C}$. 
We add to the graph induced by $A \cup \partial A$ the 
vertex $s$, and the following edges: for any $y \in C$ 
there is an edge $e_{y}$ between $y$ and $s$ with weight 
$w(e_{y}) = q_{n,y} (1 - q_{n,y})^{-1}$; 
and for any $y \in (\partial A) \setminus C$
there are edges $f_{y,1}$ and $f_{y,2}$ between $y$ and $s$,
with respective weights 
$w(f_{y,1}) = q_{n,y}(\sigma_y) (1-q_{n,y})^{-1}$
and $w(f_{y,2}) = (q_{n,y} - q_{n,y}(\sigma_y))(1 - q_{n,y})^{-1}$. 
All other edges have weight $1$. 

Observe that the weights 
have been chosen in such a way that the probability for 
the network random walk started at $y \in \partial A$ to 
reach $s$ before reaching $A$ is $q_{n,y}$, the same as it 
was in $G_n$. Let $S$ be a network random on $G_n$ stopped 
at time $\tau_n$ (the hitting time of $s$), and let 
$S'$ be a network random walk on $G'_{A,C}$, stopped at the
hitting time $\tau'$ of $s$. Let $\xi_k$ be 
the time of the $k$-th visit by $S$ to the set 
$A \cup \partial A \cup \{ s \}$. The choice of weights 
implies that if $S(0) = S'(0) \in A \cup \partial A$, then 
$\{ S(\xi_k) \}_{k \ge 0}$ has the same distribution as
$\{ S'(k) \}_{k \ge 0}$. We can couple the two walks in such 
a way that we have $S(\xi_k) = S'(k)$ for all $k \ge 0$.
Moreover, by the choice of the weights, the coupling can 
be arranged in such a way that
for each $y \in (\partial A) \setminus C$, the edge
$f_{y,1}$ is the last edge traversed by $S'$ if and only if
$\LE(S[\hat{\tau}_y,\tau_n]) = \sigma_y$, where 
$\hat{\tau}_y$ is the time of the last visit to $y$ by $S$.

Let $a = | A |$, $b = | \partial A |$ and $c = | C |$.
Let $u_1, \dots, u_{a+b}$ be an enumeration of the vertices
in $A \cup \partial A$, where 
$\{ u_1, \dots, u_{b-c} \} = (\partial A) \setminus C$,
$\{ u_{b-c+1}, \dots, u_{b+a-c} \} = A$ and
$\{ u_{b+a-c+1}, \dots, u_{b+a} \} = C$. 
Let $S^i$ and $S'^i$ be network random walks on $G_n$ and
$G'_{A,C}$, respectively, with $S^i(0) = u_i = S'^i(0)$,
coupled as above, and assume that these pairs are 
independent for $1 \le i \le a+b$.
We use Wilson's method on $G_n$ and $G'_{A,C}$ with the
above enumeration of vertices and the coupled random walks
to generate $t_n$ distributed according to $\mu_n$ and
$t'$ distributed according to $\mu_{G'_{A,C}}$.

Our assumptions and Lemma \ref{lem:sym1} imply that there 
is a unique arrow configuration in $A$ that realizes 
the event $\{ \eta_{n,x} = h_x,\, x \in A \}$, given the
restrictions $\cC(t_n) = C$, $\pi_{n,y} = \sigma_y$. 
Let $\vec{F}_A := \{ [v_1,v'_1], \dots, [v_a,v'_a] \}$ be 
this arrow configuration, where $v_j = u_{b-c+j}$.
Observe that the edges leading from $A$ to 
$(\partial A) \setminus C$ are always in $\vec{F}_A$, and 
hence we may assume without loss of generality that the 
indexing is such that $v'_1 = u_1, \dots, v'_{b-c} = u_{b-c}$. 
We define
$\vec{F}_0 = \{ f_{y,1} : y \in (\partial A) \setminus C \}$,
where these edges are oriented away from $\partial A$,
let $\vec{F}_{A,1} = \{ [v_1,u_1], \dots, [v_{b-c}, u_{b-c}] \}$, 
let $\vec{F}_{A,2} = \vec{F}_A \setminus \vec{F}_{A,1}$, and
let $\vec{F_C} := \{ e_y : y \in C \}$, where
these edges are oriented towards $C$.
The coupling ensures that the event 
in \eqref{e:event-with-paths} occurs if and only
if $t'$ consists of the edges:
\eqn{e:edges}
{ \vec{F}
  := \vec{F_0} \cup \vec{F_A} \cup \vec{F_C}. }
In order to complete the proof, we need to show that 
the probability that Wilson's method on $G'_{A,C}$
produces $\vec{F}$ is of the claimed form. 

For $i = 1, \dots, b-c$, the conditional probability 
of the event 
$\{ \cF'_i = \cF'_{i-1} \cup \{ f_{u_i,1} \} \}$
given the event 
$\{ \cF'_{i-1} = \{ f_{u_j,1} : 1 \le j < i \} \}$ 
is of the form 
\eqnst
{ f_{i,A,C} ( q_{n,y} : y \in \partial A ) 
    q_{n,u_i}(\sigma_{u_i}), }
where the form of $f_{i,A,C} : [0,1]^{\partial A} \to [0,1]$ only 
depends on the pair $(A,C)$. In particular,
\eqnst
{ p_0 
  := \P [ \cF_{b-c} = \vec{F_0} ]
  = f'_{A,C} (q_{n,y} : y \in \partial A )
     \prod_{y \in (\partial A) \setminus C} q_{n,y}(\sigma_y), }
where the form of the function $f'_{A,C}$ only depends 
on the pair $(A,C)$. Let 
\eqnsplst
{ p_{A,1}
  &:= \P [ \cF_{2(b-c)} = \vec{F}_0 \cup \vec{F}_{A,1} \,|\,
      \cF_{b-c} = \vec{F}_0 ] \\
  p_{A,2} 
  &:= \P [ \cF_{b+a-c} = \vec{F}_0 \cup \vec{F}_{A,1} 
     \cup \vec{F}_{A,2} \,|\,
     \cF_{2(b-c)} = \vec{F_0} \cup \vec{F}_{A,1} ] \\
  p_C
  &:= \P [ t' = \vec{F} \,|\, 
     \cF_{b+a-c} = \vec{F_0} \cup \vec{F_A} ]. }
Here $p_{A,1} = f''_{A,C}$, where again the form of the 
function $f''_{A,C}$ only depends on the pair $(A,C)$,
and $p_C = \prod_{y \in C} (1 - q_{n,y})$. We show that 
$p_{A,2}$ is a constant depending on $(A,C)$, and this
will complete the proof. The operation of contracting 
an edge in a graph means identifying its endpoints
to a single vertex. Let $G''_{A,C}$ denote the graph obtained
from $G'_{A,C}$ by contracting all edges 
in $\vec{F}_0 \cup \vec{F}_{A,1} \cup \vec{F}_C$. 
Conditional on the event
$\{ \vec{F}_0 \cup \vec{F}_{A,1} \cup \vec{F}_C \subset t' \}$,
the distribution of $t'$ is equal to the distribution 
of $t'' \cup \vec{F}_0 \cup \vec{F}_{A,1} \cup \vec{F}_C$,
where $t''$ is a sample from $\mu_{G''_{A,C}}$. Since 
all non-loop edges in $G''_{A,C}$ have weight $1$, 
$\mu_{G''_{A,C}}$ is uniform on $\cT_{G''_{A,C}}$.
It follows that $p_{A,2} = | \cT_{G''_{A,C}} |^{-1}$, that 
is a constant depending only on the pair $(A,C)$.

Since $\P [ t' = \vec{F} ] = p_0\, p_{A,1}\, p_{A,2}\, p_C$,
the proof of the Lemma is complete.
\end{proof}

\begin{proof}[Proof of Theorem \ref{thm:tree}] 
Let $C \subsetneqq \partial A$ and suppose that the event
$\{ \eta_{n,x} = h_x,\, x \in A \}$ is realized by
some tree $t_n$ with $\cC(t_n) = C$. For each 
$y \in (\partial A) \setminus C$, let $\sigma_y$ 
be any self-avoiding path from $y$ to $s$ that avoids $A$.
Lemma \ref{lem:sym1} shows that there is a unique 
arrow configuration $\vec{F}_A$ in $A$ (possibly 
depending on $C$ and the $\sigma_y$'s), such that 
any tree $t'_n$ with $\cC(t'_n) = C$ such that $t'_n$ 
contains the $\sigma_y$'s and $\vec{F}_A$ also
realizes the event $\{ \eta_{n,x} = h_x,\, x \in A \}$.
Lemma \ref{lem:sym2} shows that the probability mass
of all such trees is given by the expression 
on the right hand side of \eqref{e:event-with-paths}.
It follows that we have
\eqn{e:1stexpr}
{ \nu_{G_n} [ \eta_{n,x} = h_x,\, x \in A ]
  = \sum_{C \subsetneqq \partial A} 
    \sum_{ \{ \sigma_y : y \in (\partial A) \setminus C\} }
    f_{A,C}( q_{n,y} : y \in \partial A )
    \prod_{y \in (\partial A) \setminus C} q_{n,y} ( \sigma_y ). }
It is crucial here that the second sum is over \emph{all}
collections self-avoiding paths from the $y$'s to $s$,
avoiding $A$, and that $f_{A,C}$ is independent of the paths.

Observe that $\sum_{\sigma_y} q_{n,y}(\sigma_y) = q_{n,y}$, and 
therefore, performing the sum over the $\sigma_y$'s 
in \eqref{e:1stexpr} gives
\eqnst
{ \nu_{G_n} [ \eta_{n,x} = h_x,\, x \in A ]
  = \sum_{C \subsetneqq \partial A} 
    f_{A,C}( q_{n,y} : y \in \partial A )
    \prod_{y \in (\partial A) \setminus C} q_{n,y}. }
Regardless of what the exhaustion is, we have
\eqnst
{ \lim_{n \to \infty} q_{n,y}
  = \P [ S[0,\infty) \cap A = \es \,|\, S(0) = y ] 
  = \frac{d-2}{d-1}. }
Hence we have
\eqnst
{ \nu [ \eta_{n,x} = h_x,\, x \in A ]
  = \sum_{C \subsetneqq \partial A} 
    f_{A,C} \left( \frac{d-2}{d-1}, \dots, 
    \frac{d-2}{d-1} \right)
    \left( \frac{d-2}{d-1} \right)^{|(\partial A) \setminus C|}. }
\end{proof}

\begin{remark}
As an example, taking $A = \{ o \}$, one can recover the 
computation of height probabilities by 
Dhar and Majumdar \cite{DM90} from the above.
\end{remark}

\section{Concluding remarks}
\label{sec:conclusion}

\subsection{Finiteness of avalanches}
\label{ssec:finiteness}

Following the program introduced by Maes, Redig and Saada \cite{MRS02},
once the existence of the limit $\nu$ has been established,
it is natural to ask if one can define sandpile dynamics
on the infinite graph $G$. The first question is whether
adding a particle at a vertex $o$ in a sample configuration 
from the measure $\nu$ produces an avalanche that is 
finite $\nu$-a.s.~(that is, only finitely many topplings are 
necessary to stablilize it). In \cite{JR08} a sufficient condition 
was given, in the case of $\Z^d$, $d \ge 3$, in terms of a 
certain modification 
$\tilde{\mu}$ of the measure $\mu$. Let $\tilde{G}_n$ be the 
graph obtained from $G_n$ by wiring $o$ to the sink, and let 
$\tilde{\mu} = \lim_{n \to \infty} \mu_{\tilde{G}_n}$ be the limiting 
wired spanning forest measure. Let $t_o$ denote the component
of $o$ under the measure $\tilde{\mu}$. It was shown in \cite{JR08}
that if $\tilde{\mu} [ | t_o | < \infty ] = 1$, then avalanches 
on $G$ are finite $\nu$-a.s.

The arguments in \cite{JR08} apply without change to show 
that on any transient graph, avalanches are $\nu$-a.s.~finite,
if $\tilde{\mu} [ | t_o | < \infty ] = 1$. 
Lyons, Morris and Schramm \cite{LMS08} analyzed $t_o$ under
general conditions, in particular have shown that $t_o$ is
finite in any transitive graph with at least cubic volume 
growth. They have also shown that finiteness of $t_o$
is equivalent to the one-end property. Since we assumed the 
one-end property in Theorems \ref{thm:AJlow} and \ref{thm:trans}, 
it follows that the limiting measures constructed in these
theorems have a.s.~finite avalanches, if the graph is 
transient.

Let us also discuss the case of the $d$-regular tree. 
Let $o \in T^d$ be a fixed vertex. Take a sample configuration
from the measure $\nu$, and add a particle at $o$. Let $N$ 
denote the number distinct vertices that topple in the
stabilization of this configuration. The computations 
in \cite{DM90} show that $\nu [ N = n ] \sim c n^{-3/2}$ as
$n \to \infty$, and also that $\sum_{n=0}^\infty \nu [ N = n ] = 1$.
Hence avalanches are $\nu$-a.s.~finite. As above, finiteness 
can also be derived from the well-known fact that each tree in 
the WSF has one end \cite{Hagg98} 
(see also \cite[Section 11]{BLPS01}). 

\begin{question}
Are avalanches $\nu$-a.s.~finite on $\Z^2$? More generally, are
avalanches $\nu$-a.s.~finite on a recurrent graph $G$ such that 
the WSF has one end a.s.? See \cite{JL07} for a related open 
question regarding a weaker property.
\end{question}

\subsection{Stationary Markov process}
\label{ssec:process}

Having established finiteness of avalanches, one can apply 
the general machinery developed in \cite{MRS02} to show 
the existence of a natural stationary Markov process
with invariant measure $\nu$. Suppose that $G = (V,E)$ is
transient. Let $\varphi : V \to (0,\infty)$
be a function such that $\sum_{x \in V} \varphi(x) G(x,o) < \infty$,
where $G = \Delta^{-1}$. Given an exhaustion 
$V_1 \subset V_2 \subset \dots \subset V$, consider the
continuous time sandpile Markov chain on $G_n$, where 
particles are added at $x \in V_n$ at Poisson rate $\varphi(x)$.
This Markov chain has invariant measure $\nu_{G_n}$, 
and it follows from the results of \cite{MRS02} that its
semigroup strongly converges in $L^2(\nu)$ to the semigroup 
of a Markov process with invariant measure $\nu$.

\appendix

\section{Appendix}
\label{A:mixing}

In this appendix we give the proof that the path-valued
process $\{ X_k \}_{k \ge 0}$ introduced in the proof of 
Lemma \ref{lem:goodblocks} is stationary and mixing.
We will write $p(x,y)$ for the transition probability
of $S$.

Let $[x_k(-m), \dots, x_k(-1), x_k(0), x_k(1), \dots, x_k(m)]$ 
be fixed finite paths in $G$, for $k = 0, \dots, K-1$, such that 
$x_k(0) = o$. Without loss of generality, we assume that $m > M$. 
Call two finite paths 
$y_1 = [y_1(-\ell_1), \dots, y_1(0) = o, \dots, y_1(\ell_2)]$ and 
$y_2 = [y_2(-\ell_1), \dots, y_2(0) = o, \dots, y_2(\ell_2)]$
\emph{equivalent}, $y_1 \equiv y_2$, if there exists 
$\bar\psi \in \AUT_o$ such that 
$\bar\psi y_1(j) = y_2(j)$ for each $j$.
This is clearly an equivalence relation.

We will use the following simple lemma, whose proof is
obvious.

\begin{lemma}
\label{lem:inv}
Suppose that $T$ is a transformation from a finite set of
paths $\cP_1$ into a finite set of paths $\cP_0$, where
$\cP_0$ and $\cP_1$ have the same number of elements.
Suppose that $T$ has the propoerty that 
$y_1 \equiv y_2$ if and only if $Ty_1 \equiv Ty_2$. 
Let $\la_i : \cP_i \to \R$, $i = 0,1$ be functions 
that are constant on equivalence classes, such that 
$\la_1(y) = \la_0(Ty)$. Then
\eqnst
{ \sum_{y \in \cP_1} \la_1(y)
  = \sum_{y' \in \cP_0} \la_0(y). }
\end{lemma}

Consider the probability 
\eqnspl{e:X0rewrite}
{ &\P \left[ X_k(j) = x_k(j),\, 
     -m \le j \le m,\, 0 \le k \le K-1 \right] \\
  &\qquad = 
    \P \left[ \Psi_k \phi^{-1}_{S(kM)} S(kM+j) = x_k(j),\, 
     -m \le j \le m,\, 0 \le k \le K-1 \right] \\
  &\qquad =
    \sum_{y'}
    \P \left[
    S(j) = y'(j),\, -m \le j \le (K-1)M+m \right] \\
  &\qquad\qquad \times
    \prod_{k=0}^{K-1} 
    \P \left[ \Psi_k \phi^{-1}_{y'(kM)} y'(kM+j) = x_k(j),\, 
    -m \le j \le m \right], }
where the summation is over all paths $y'$ with parameter set
$\{ -m, \dots, (K-1)M + m \}$ that are at $o$ at time $0$.
The first factor in the right hand side of \eqref{e:X0rewrite}
equals
\eqnst
{ \prod_{j = -m}^{(K-1)M+m-1} p(y'(j),y'(j+1)). }
In order to abreviate the second factor, introduce 
the notation
$U_k y'(j) = \phi^{-1}_{y'(kM)} y'(kM+j)$, $-m \le j \le m$.
Then the right hand side of \eqref{e:X0rewrite} is
\eqnsplst
{ \la_0(y')
  &:= \prod_{j = -m}^{(K-1)M+m-1} p(y'(j),y'(j+1)) \\
  &\qquad \times \prod_{k=0}^{K-1} 
     \P \left[ \Psi_k U_k y' (j) = x_k(j),\, 
     -m \le j \le m \right]. }

Now consider
\eqnspl{e:X1rewrite}
{ &\P \left[ X_k(j) = x_{k-1}(j),\, 
     -m \le j \le m,\, 1 \le k \le K \right] \\
  &\qquad = 
    \P \left[ \psi_k \phi^{-1}_{S(kM)} S(kM+j) = x_{k-1}(j),\, 
     -m \le j \le m,\, 1 \le k \le K \right] \\
  &\qquad =
    \sum_{y}
    \P \left[
    S(j) = y(j),\, -m+M \le j \le KM+m \right] \\
  &\qquad\qquad \times 
    \prod_{k=1}^{K} 
    \P \left[ \Psi_k \phi^{-1}_{y(kM)} y(kM+j) = x_{k-1}(j),\, 
    -m \le j \le m \right] \\
  &\qquad =: \sum_y \la_1(y), }
where the summation is over paths $y$ with parameter set
$\{ -m+M, \dots, KM + m \}$ that are at $o$ at time $0$.
Introduce the map $T$, $Ty(j) = \phi^{-1}_{y(M)} y(M+j)$.
The first factor in the right hand side of \eqref{e:X1rewrite}
equals
\eqn{e:walkinv}
{ \prod_{j = -m+M}^{KM+m-1} p(y(j),y(j+1))
  = \prod_{j = -m}^{(K-1)M+m-1} p(Ty(j),Ty(j+1)). }
The second factor equals
\eqn{e:productoverk}
{ \prod_{k=0}^{K-1} 
  \P \left[ \Psi_{k+1} \phi^{-1}_{y((k+1)M)} y((k+1)M + j)
      = x_k(j),\, -m \le j \le m \right]. }
We claim that the path $\{ \phi^{-1}_{y((k+1)M)} y((k+1)M+j) \}_{j=-m}^m$
is equivalent to the path $U_k Ty$. Indeed, it is easy
to check that the automorphism 
$\tilde{\psi} = \phi^{-1}_{Ty(kM)} \phi^{-1}_{y(M)} \phi_{y((k+1)M)}$ 
does the job. Hence, using right invariance of $\la$,
the $k$-th factor in \eqref{e:productoverk} equals
\eqnsplst
{ &\P \left[ \Psi_{k+1} \tilde{\psi}^{-1} U_k Ty = x_k \right] 
  = 
  \la \left( \left\{ \Psi : \Psi \tilde{\psi}^{-1} U_k Ty = x_k 
     \right\} \right) \\
  &\qquad = 
  \la \left( \left\{ \Psi : \Psi U_k Ty = x_k 
     \right\} \tilde{\psi} \right) \\
  &\qquad = 
  \la \left( \left\{ \Psi : \Psi U_k Ty = x_k 
     \right\} \right) \\
  &\qquad =
  \P \left[ \Psi_k U_k Ty = x_k \right]. }
This and \eqref{e:walkinv} shows that $\la_1(y) = \la_0(Ty)$. 
A similar computation shows that $\la_0$ is 
constant on equivalence classes.

It is left to show that $y_1 \equiv y_2$ if and only if
$Ty_1 \equiv Ty_2$. Indeed, $y_2 = \bar{\psi} y_1$ 
if and only if $Ty_2 = \phi^{-1}_{y_2(M)} \bar{\psi} \phi_{y_1(M)} Ty_1$.
An application of Lemma \ref{lem:inv} now shows that
the expressions in \eqref{e:X0rewrite} and
\eqref{e:X1rewrite} equal, and this is sufficent
to conclude stationarity.

The proof of mixing can be carried out using a similar 
computation.
Suppose that $(K-1)M + m < tM - m$, and consider the
probability:
\eqnspl{e:X0Xtrewrite}
{ &\P \left[ X_k(j) = x_k(j),\, 
    -m \le j \le m,\, k \in \{0, \dots, K-1\} 
    \cup \{t \dots, t+K-1\} \right] \\
  &=
    \sum_{y_0} \sum_{y} \sum_{y_t}
    p(y_0) p(y) p(y_t) \nu( y_0 ) \\
  &\qquad \times
    \prod_{k=0}^{K-1} 
    \P \left[ \Psi_{t+k} \phi^{-1}_{y_t((t+k)M)} y_t((t+k)M+j) 
    = x_{t+k}(j),\, 
    -m \le j \le m \right]. }
Here the $y_0$-sum is over paths with parameter set 
$\{ -m, \dots, (K-1)M + m \}$ that are at $o$ at time $0$,
the $y$-sum is over paths with parameter set 
$\{ (K-1)M + m, \dots, tM - m \}$  
starting at $y_0((K-1)M+m)$, and the $y_t$-sum is over
paths with parameter set $\{ tM-m, \dots, (t+(K-1))M + m \}$
that start at $y(tM-m)$. The expressions $p(y_0)$, $p(y)$ 
and $p(y_t)$ stand for the products of random walk transition 
probabilities for these paths, and $\nu(y_0)$ is the expression 
containing $\Psi_0, \dots, \Psi_{K-1}$. 
Keeping $y_0$ and $y$ fixed, we introduce the
map $T$, $Ty_t (i) = \phi^{-1}_{y_t(tM)} y_t (tM + i)$,
$-m \le i \le (K-1)M + m$. Then $T$ maps paths starting
at $y(tM-m)$ to paths with parameter set 
$\{ -m, \dots, (K-1)M+m \}$ that are at $o$ at time $0$.
A straightforward computation then shows that
$\phi^{-1}_{y_t((t+k)M)} y_t((t+k)M + \cdot) \equiv
U_k Ty_t(\cdot)$. An application of Lemma \ref{lem:inv}
yields that (keeping $y_0$ and $y$ fixed)
\eqnsplst
{ &\sum_{y_t} p(y_t) 
    \prod_{k=0}^{K-1} 
    \P \left[ \Psi_{t+k} \phi^{-1}_{y_t((t+k)M)} y_t((t+k)M+j) 
    = x_{t+k}(j),\, 
    -m \le j \le m \right] \\
  &\qquad = \sum_{y'} \tilde{\la}_0(y'), }
where $\tilde{\la}_0$ has the same form as $\la_0$, with
$x_k$ replaced by $x_{t+k}$. We can now carry out 
the  summation over $y$ to yield a factor $1$, and
conclude that the expression in \eqref{e:X0Xtrewrite}
equals
\eqnspl{e:X0Xtindep}
{ &\P \left[ X_k(j) = x_k(j),\, 
    -m \le j \le m,\, 0 \le k \le K-1 \right] \\
  &\qquad \times 
  \P \left[ X_{t+k}(j) = x_{t+k}(j),\, 
    -m \le j \le m,\, 0 \le k \le K-1 \right]. }
The equality in \eqref{e:X0Xtindep} is sufficient to
conclude mixing.

\medbreak

\end{document}